\documentclass[12pt]{article}
\usepackage{a4wide}
\usepackage{amstext}
\usepackage{amsmath} 
\usepackage{amssymb}
\usepackage{graphicx} 
\usepackage{array}
\usepackage{epsfig}
\usepackage{todonotes}
\usepackage{stmaryrd}
\usepackage[update,prepend]{epstopdf}

\numberwithin{equation}{section}

\newcommand{\mcF}{\mathcal{F}}
\newcommand{\mcG}{\mathcal{G}}

\newcommand{\mcI}{\mathcal{I}}

\newcommand{\mcK}{\mathcal{K}}

\newcommand{\mcO}{\mathcal{O}}

\newcommand{\tn}{|\mspace{-1mu}|\mspace{-1mu}|}
\newcommand{\IR}{\mathbb{R}}

\newcommand{\nablas}{\nabla_\Omega}

\newtheorem{lem}{Lemma}[section]

\newtheorem{thm}{Theorem}[section]
\newtheorem{rem}{Remark}[section]

\newenvironment{proof}{\noindent \newline {\bf Proof.}}
{\hfill \mbox{\fbox{} } \newline}
\usepackage{enumitem}
\usepackage[margin=10pt,font=small,labelfont=bf,labelsep=endash]{caption}
\usepackage{subcaption}
\usepackage{cite}
\usepackage{hyperref}

\begin{document}
\title{\bf A Nitsche Method for Elliptic Problems 
on Composite Surfaces\thanks{This research was supported in part by the Swedish Foundation for Strategic Research Grant No.\ AM13-0029, the Swedish Research Council Grants Nos.\ 2011-4992, 2013-4708, the Swedish Research Programme Essence}}
\author{Peter Hansbo
\footnote{Department of Mechanical Engineering, J\"onk\"oping University, 
SE-55111 J\"onk\"oping, Sweden.} 
\mbox{ }
Tobias  Jonsson
\footnote{Department of Mathematics and Mathematical Statistics, Ume{\aa} University, SE-90187 Ume{\aa}, Sweden} 
\mbox{ }
Mats G.\ Larson
\footnote{Department of Mathematics and Mathematical Statistics, Ume{\aa} University, SE-90187 Ume{\aa}, Sweden} 
\mbox{ }
Karl Larsson
\footnote{Department of Mathematics and Mathematical Statistics, Ume{\aa} University, SE-90187 Ume{\aa}, Sweden} 
}
\date{}
\numberwithin{equation}{section} \maketitle
\begin{abstract}
We develop a finite element method for elliptic partial differential equations 
on so called composite surfaces that are built up out of a finite number of surfaces 
with boundaries that fit together nicely in the sense that the intersection between any two surfaces in the composite surface is either empty, a point, or a curve segment, called an interface curve. Note that several surfaces can intersect along the same interface curve. On the composite surface we consider a broken finite element space which consists of a continuous finite element space at each subsurface without continuity requirements across the interface curves. We derive a Nitsche type formulation in this general setting and by assuming only that a certain inverse inequality and an approximation property hold we can derive stability and error estimates in the case when the geometry is exactly represented. We discuss several different realizations, including so called cut meshes, of the method. Finally, we present numerical examples.
\end{abstract}

\clearpage
\tableofcontents
\clearpage

\section{Introduction}

\paragraph{Background.}
Many physical phenomena takes place on geometries that consist of an arrangement 
of surfaces, for instance transport of surfactants, heat transfer, and flows in 
cracks. See Figure \ref{fig:schematic} for examples of surface arrangements.
In manufacturing the use of surface arrangements to minimize the amount of needed material, for example in the form of honeycomb sandwich structures, is well established while recent 
developments of additive manufacturing enable production of even more complex surface structures.
The arrangement of surfaces in applications often contain sharp edges, 
corners, and lines where several surfaces meet. Thus there is significant interest 
in the development of finite element methods for solving partial differential 
equations on such general geometries.

\begin{figure}[h]
    \centering
    \begin{subfigure}[b]{0.45\textwidth}
        \centering
        \includegraphics[width=0.8\textwidth]{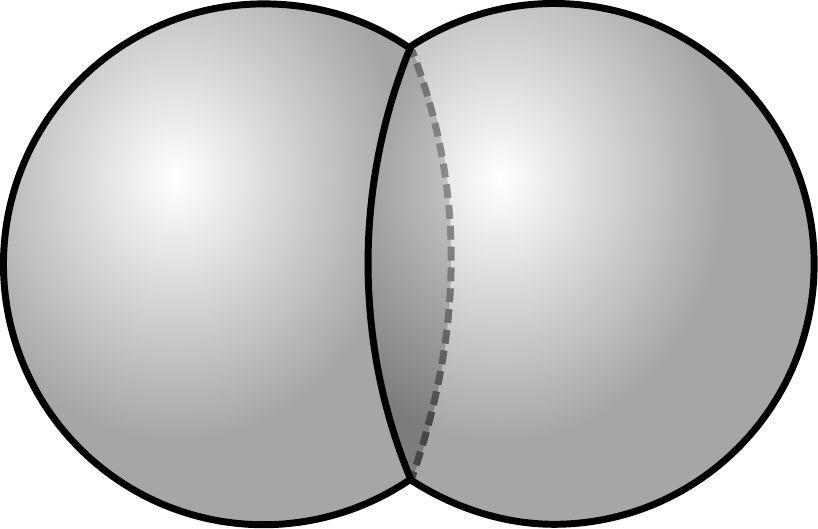}
	\vspace{0.7em}        
        
       \caption{Intersecting bubbles}
    \end{subfigure}
    \begin{subfigure}[b]{0.45\textwidth}
        \centering
        \includegraphics[width=0.85\textwidth]{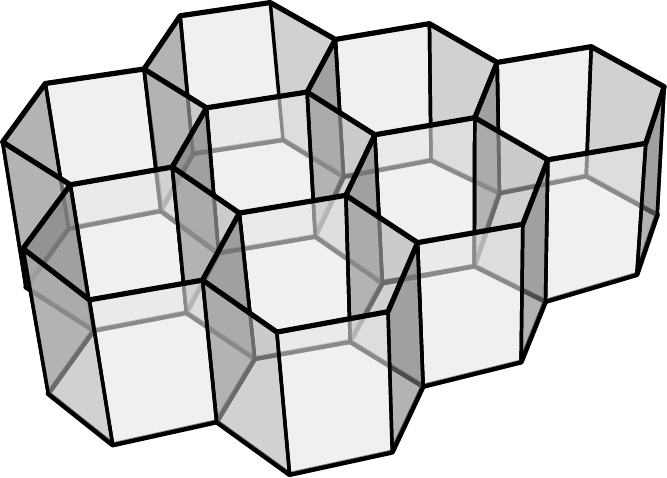}
        \caption{Honeycomb structure}
    \end{subfigure}
\caption{Illustrations of composite surface configurations. }
\label{fig:schematic}
\end{figure}

\paragraph{New Contributions.}
In this contribution we develop a Nitsche method for a diffusion 
problem on a such an arrangement of surfaces. The key feature is that the 
formulation can handle interfaces where several surfaces meet at intersecting 
interfaces, including triple points and sharp edges. The method avoids defining 
a conormal to each interface and instead the 
well defined conormal associated with each subsurface is used together with the 
natural conservation law: the sum of all conormal fluxes is zero at the interface. 
This conservation law is sometimes refereed to as the Kirchhoff condition. We show 
that the method is equivalent to the standard Nitsche interface method in flat geometries. 
The same idea naturally extend to discontinuous Galerkin methods on surfaces, where 
instead of defining a conormal to each edge which would be needed in a standard 
discontinuous Galerkin method, see for instance the discussion in \cite{DedMad13}, 
the well defined discrete element conormal can be used.

We consider different ways of constructing a mesh on a composite geometry 
including matching meshes, non matching meshes, and cut meshes. For cut meshes 
we add a stabilization term that provides control in the vicinity of the interfaces. More 
precisely we consider a stabilization term, see \cite{BurCla15} and \cite{JonLarLar17}, that 
satisfies certain abstract conditions which enable us to prove discrete stability and 
a priori error estimates. For clarity we restrict our analysis to the case when the geometry 
is exactly represented. This can be realized using parametric mappings, see \cite{JonLarLar17}.
We also give a concrete construction of such a stabilization term based on penalization 
of jumps in derivatives across faces belonging to elements that intersect the interface. We 
conclude the paper with some illustrating numerical examples.

\paragraph{Earlier Work.}
Since the pioneering work of Dziuk \cite{Dzi88} where a continuous Galerkin method 
for the Laplace-Beltrami operator on a triangulated surface was first proposed 
there has been several extensions including adaptive and higher order  methods,  
Demlow and Dziuk \cite{DemDzi07} and Demlow \cite{Dem09}, and higher order problems, Clarenz et al. \cite{CalDie04} 
and Larsson and Larson \cite{LarLar13}. A standard 
discontinuous Galerkin method for the Laplace--Beltrami operator on a 
smooth closed surface was analyzed in \cite{DedMad13}. For further extensions 
including time dependent problems we also refer to the review articles 
\cite{DecDzi05, DziEll13a}. Models of membranes were considered in \cite{HanLar14} and 
\cite{HanLarLar15}.
All of these contributions deals with smooth 
surfaces and the discontinuous Nitsche formulation proposed in this paper 
which allows more complex surfaces appears to be new. In \cite{HanLar17} we 
develop a method for plate structures on composite surfaces consisting of plane 
surfaces with the restriction that only two plates meet at an interface. Here 
each plate is modeled using a membrane model and a fourth order Kirchhoff model 
for the bending.
Various methods for connecting parametric patches pairwise which also allow for 
cut meshes have been proposed, see for example \cite{Langer2015,RueSch14,KolOzc15}
and the extension to surfaces in \cite{JonLarLar17}.

\paragraph{Outline.}
The outline of the remainder of the paper is as follows. In Section~\ref{sec:modelprob} we give a short introduction to tangential calculus on surfaces and then we formulate the model problem. 
In Section~\ref{sec:method} we derive the method, study how it relates to the standard discontinuous Galerkin 
method in the the case of an interface in flat geometry, and introduce the stabilization term.
In Section~\ref{sec:analysis} we prove a priori error estimates in the energy and $L^2$ norm. Finally, in Section~\ref{sec:numerical} 
we present some numerical examples illustrating the method on three different test cases.

\section{Model Problem}\label{sec:modelprob}

\subsection{The Composite Surface}\label{sec:surfaceassumptions}

\paragraph{Surface with Boundary.} Let $\Sigma$ be  
smooth surface embedded in $\IR^3$ with orientation (or normal) 
$n$ and boundary $\partial \Sigma$, which has the following 
properties
\begin{itemize}  
\item There is a smooth closed surface $\widetilde{\Sigma}$ embedded 
in $\IR^d$ such that $\Sigma \subset \widetilde{\Sigma}$.
\item The boundary $\partial \Sigma$ consists of a finite set of smooth 
curve segments and corner points $N(\partial \Sigma)$. The exterior 
unit conormal to $\partial \Sigma$ is denoted by $\nu_{\partial \Sigma}$.
\item At each corner $x \in N(\partial \Sigma)$ on the boundary  there is 
a constant such that 
\begin{equation}
-1 < C \leq \nu^+ (x) \cdot \nu^-(x)
\end{equation}
where $\nu^{\pm}(x)$ denotes the left and right conormal 
at the corner $x$.
\end{itemize}

\paragraph{Composite Surface.}
We introduce the following notation:
\begin{itemize}
\item Let $\mcO = \{ \Omega_i : i \in \mcI_\Omega\}$ be a set of smooth 
surfaces with boundaries that satisfy the assumptions above.

\item A 
composite surface $\Omega$ is a finite union of surfaces with 
boundaries 
\begin{equation}
\Omega = \bigcup_{i \in \mcI_\Omega} \Omega_i
\end{equation}
\item The intersection between any two surfaces $\Omega_{i}$
and $\Omega_{j}$, $i,j\in \mcI_\Omega$, is either empty 
or occur at the boundary of the surfaces, 
\begin{equation}\label{eq:pairwiseintersection}
\Omega_{i} \cap \Omega_{j} 
= \partial \Omega_{i} \cap \partial \Omega_{j} 
= \Gamma_{ij}
\end{equation}
and $\Gamma_{ij}$ is either a smooth curve segment (an interface 
curve) or a point (an interface node). Furthermore, given two surfaces 
$\Omega_i$ and $\Omega_j$, there is a sequence of surfaces which starts 
with $\Omega_i$ and end at $\Omega_j$ such that two consecutive 
surfaces share a nonempty interface. This means that no surface share only 
a point with the union of the other surfaces.

\item We 
may also assume that $\Gamma_{ij}$ is one of the curve segments in 
$\partial \Omega_i$ and $\partial \Omega_j$, if not 
we may simply modify the boundary description.

\item Let $\mcG_I = \{ \Gamma_k : k \in \mcI_{\Gamma_I}\}$ be the set of all non-overlapping interface curves, $\Gamma_k\cap\Gamma_l=\emptyset$ for $k\neq l$, such that
\begin{align}
\bigcup_{k\in\mcI_{\Gamma_I}} \Gamma_k
=
\bigcup_{i,j\in\mcI_\Omega} \Gamma_{ij}
\end{align}
Thus the union of the non-overlapping interface curves in $\mcG_I$ includes all interface curves \eqref{eq:pairwiseintersection}.

\item 
Given an index $j \in \mcI_{\Gamma_I}$ we let  $\mcI_\Omega(j) \subseteq \mcI_\Omega$ 
be the set of indexes corresponding to surfaces that share interface 
$\Gamma_j$.

\end{itemize}

\subsection{Elliptic Model Problem}

\paragraph{Notation.}
We introduce the following notation:
\begin{itemize}
\item The tangential gradient $\nablas$ is defined by 
$\nablas v = P_\Omega \nabla v$, where $P_\Omega = I - n \otimes n$ is the 
projection onto the tangent plane of $\Omega$.
\item $\mu:\Omega \rightarrow \IR$ is a given 
function such that $\mu|_{\Omega_i} = \mu_i$ with $\mu_i\in C(\Omega_i)$, 
and there are constants such that for all $x\in \Omega$ it holds
\begin{equation}
0< c \leq \mu(x) \leq C 
\end{equation}

\item The flux is defined by
\begin{equation}
\sigma (u) = \mu \nablas u   
\end{equation}
\end{itemize}

\paragraph{Model Problem.}
Consider the problem: find $u : \Omega \rightarrow {\IR}$ such that
\begin{align}\label{eq:prob-interior}
-\nablas \cdot \sigma( u )  &= f \quad \text{in $\Omega_i$, $i\in \mcI_\Omega$}
\end{align}
and that the following interface and boundary conditions are satisfied:
\begin{itemize}
\item
For each $j\in \mcI_{\Gamma_I}$ the following interface 
conditions hold
\begin{alignat}{3}\label{eq:prob-interface-flux}
\sum_{k \in \mcI_\Omega(j)} 
\nu_{k} \cdot \sigma_k(u_k) &= 0 & \qquad & \text{on $\Gamma_j$}
\\ \label{eq:prob-interface-cont}
u_{k} &= u_{l} & \qquad & \text{on $\Gamma_j$,} 
\qquad k,l \in \mcI_{\Omega}(j)
\end{alignat}
The first condition is a so called Kirchhoff condition and corresponds to 
conservation over interfaces while the second condition corresponds to continuity at the interface. Note that \eqref{eq:prob-interface-flux} thus encompasses the case where the interface is a sharp edge such as illustrated in Figure~\ref{fig:conormals}.

\item
The following boundary conditions hold
\begin{alignat}{2}
\nu\cdot\sigma(u) &= g_N &\qquad& \text{on $\Gamma_N$}
\\
u &= g_D &\qquad& \text{on $\Gamma_D$}
\end{alignat}
where the boundary $\partial\Omega$ is decomposed into a Neumann boundary $\Gamma_N$ and a Dirichlet boundary $\Gamma_D$ such that $\partial\Omega=\Gamma_N\cup\Gamma_D$, $\Gamma_N\cap\Gamma_D=\emptyset$ and $\Gamma_D\neq\emptyset$.
\end{itemize}

\begin{figure}
\centering
\includegraphics[width=5cm]{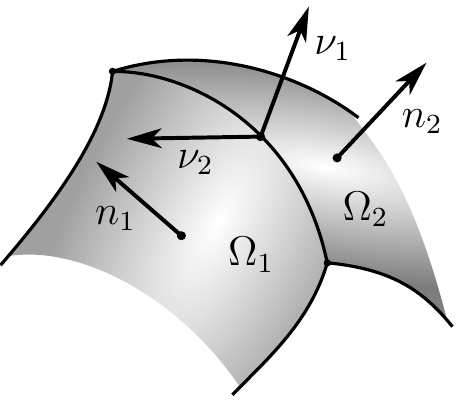}
\caption{Interface between subsurfaces $\Omega_1$ and $\Omega_2$ constituting
a sharp edge on $\Omega$, with conormals $\nu_1$ and $\nu_2$, 
and normals $n_1$ and $n_2$.}
\label{fig:conormals}
\end{figure}

\subsection{Weak Form}
\paragraph{Function Spaces.}
Let
\begin{align}
V_{g_D}(\Omega) =\left\{ v:\Omega \rightarrow \IR \, : \,
\begin{alignedat}{3}
&v_i \in H^1(\Omega_i), \quad &&\forall i \in \mcI_\Omega 
\\
&v_k = v_l \text{ on $\Gamma_j$}, \quad &&k,l \in \mcI_\Omega(j) , \ \forall j \in \mcI_{\Gamma_I}
\\
&v = g_D \text{ on $\Gamma_D$}
\end{alignedat}
\right\}
\end{align}
be equipped with the $H^1$ inner product and associated norm
\begin{equation}
(v,w)_{H^1(\Omega)} = \sum_{i\in \mcI_\Omega} 
(\nabla_{\Omega_i} v_i , \nabla_{\Omega_i} v_i)_{\Omega_i} 
+ (v,w)_{\Omega_i},
\qquad
\| v \|^2_{H^1(\Omega)} = \sum_{i\in \mcI_\Omega} \| v_i \|^2_{H^1(\Omega_i)}
\end{equation}
Then $V_0$ is clearly a Hilbert space.

\paragraph{Derivation of Weak Form.}
Starting from (\ref{eq:prob-interior}) and using Green's formula 
on each of the surfaces $\Omega_i$ in $\Omega$ we obtain for 
$v \in V_0$, 
\begin{align}
(f,v)_{\Omega} 
&= \sum_{ i\in \mcI_\Omega}(f,v)_{\Omega_i}
\\
&= \sum_{ i\in \mcI_\Omega} (-\nablas \cdot \sigma (u), v)_{\Omega_i} 
\\
&= 
\sum_{i \in \mcI_\Omega} (\sigma(u), \nablas v)_{\Omega_i} 
- (\nu \cdot \sigma(u),v)_{\partial \Omega_i}
\\
&= 
\sum_{i \in \mcI_\Omega} (\sigma(u), \nablas v)_{\Omega_i} 
- (\nu \cdot \sigma(u),v)_{\Gamma_N}
-
\sum_{j \in \mcI_{\Gamma_I}} \sum_{k\in \mcI_\Omega(j)}
(\nu_k \cdot \sigma_k(u_k),v_k)_{\Gamma_j}
\\
&= \sum_{i \in \mcI_\Omega} (\sigma(u), \nablas v)_{\Omega_i} 
- (g_N, v)_{\Gamma_N}
\end{align}
where we used the identity $v=0$ on $\Gamma_D$
 and 
$\nu\cdot \sigma(u) = g_N$ on $\Gamma_N$
 as well as the continuity 
$v_k|_{\Gamma_j} = v_l|_{\Gamma_j} = v$ for any 
$k,l \in \mcI_{\Omega}(j)$ and the interface condition (\ref{eq:prob-interface-flux}) 
to conclude that 
\begin{equation}
\sum_{k\in \mcI_\Omega(j)}
(\nu_k \cdot \sigma_k(u_k),v_k)_{\Gamma_j} 
=
\left( \sum_{k\in \mcI_\Omega(j)}
\nu_k \cdot \sigma_k(u_k),v\right)_{\Gamma_j} = 0
\end{equation}
Thus we arrive at the following weak problem: find $u \in V_{g_D}$ such that 
\begin{equation}\label{eq:exactsol}
a(u,v) = l(v), \quad \forall v \in V_0
\end{equation}
where 
\begin{equation}
a(u,v) = \sum_{i \in \mcI_\Omega} 
(\sigma(u), \nablas v)_{\Omega_i}, \qquad l(v) = \sum_{i \in \mcI_\Omega} 
(f,v)_{\Omega_i} + (g_N, v)_{\Gamma_N}
\end{equation}
\paragraph{Existence and Uniqueness.}
The form $a(\cdot,\cdot)$ is coercive in $V_0$ and thus in case $g_D = 0$ existence and 
uniqueness of the solution $u\in V_0$ follows directly from the Lax-Milgram lemma. If $g_D \neq 0$ 
we let $u_{g_D} \in V_{g_D}$ be an extension of $g_D$, set $u = u_0 + u_{g_D}$, where $u_0\in V_0$ 
is determined by $a(u_0,v) = l(v) - a(u_{g_D},v)$ for all $v\in V_0$. Here we can again apply the 
Lax-Milgram lemma.

\paragraph{Regularity of the Solution.}
As the assumptions in Section~\ref{sec:surfaceassumptions} allow for very intricate surface configurations it is challenging to give any precise prediction on the regularity of the solution $u$.
However, motivated by the regularity properties of elliptic problems in planar domains with nonsmooth boundary, see for example \cite{Grisvard85}, we assume that
\begin{equation}\label{eq:regularity}
\sum_{i \in \mcI(\Omega_i)} \| u \|^2_{H^{\eta_i}(\Omega_i)} 
\lesssim 
 \sum_{i \in \mcI(\Omega_i)} \| f \|^2_{H^{\eta_i-2}(\Omega_i)\cap L^2(\Omega_i)} 
\end{equation}
where $ \| f \|^2_{H^{\eta_i-2}(\Omega_i)\cap L^2(\Omega_i)} = 
\max( \| f \|_{H^{\eta_i-2}(\Omega_i)} ,\|f \|_{L^2(\Omega_i)})$. 
%

\section{The Finite Element Method}\label{sec:method}

In this section we derive our finite element method on the composite 
surface $\Omega$. For clarity we consider the situation when 
$\Omega$ is exactly represented, which may be realized using exact 
parametric mappings. In the case of domains described by CAD models, 
the underlying assumption in isogeometric analysis, see 
\cite{HuCoBa05,IGABook}, the surface is equipped with a parametric finite element space consisting 
of continuous functions without any continuity requirement across 
the interface curves. In this setting we can focus on essentials and 
derive the method together with the basic properties including a 
stability result and an priori error estimate. 

\subsection{Constructions of Meshes on Composite Surfaces}
There are several different natural ways to construct a mesh on a composite 
surface:
\begin{enumerate}[label=(\alph*)]
\item Each surface is meshed with elements and matching meshes are used 
across the interface curves.
\item Each surface is meshed with elements and non matching meshes 
are used across the interface curves.
\item A number of surfaces are individually meshed and arranged in 
such a way that they intersect. In this situation so called cut elements 
naturally occur close to the interface.
\item Each of the surfaces is meshed using a cut finite element technique.   
\end{enumerate}
Examples of these mesh constructions are illustrated in Figure~\ref{fig:mesh-constructions}.
In the case of cut elements we add a stabilization term which enables us 
to prove stability and optimal order error estimates.

\begin{figure}
    \centering
    \begin{subfigure}[b]{0.3\textwidth}
        \centering
        \includegraphics[width=0.65\textwidth]{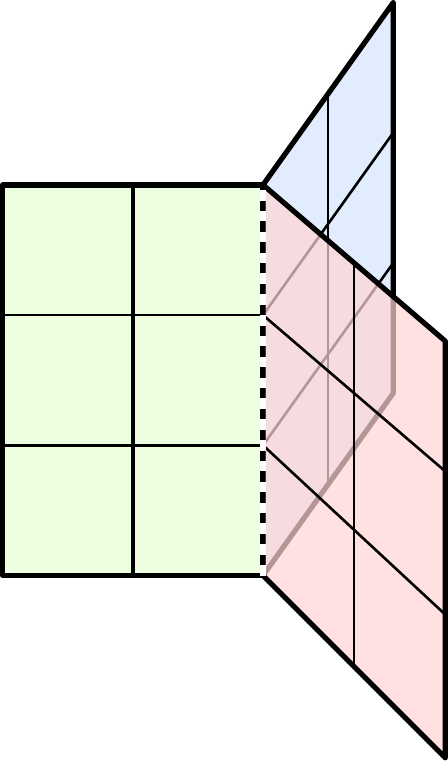}
       \caption{Matching grids}
    \end{subfigure}
    \begin{subfigure}[b]{0.3\textwidth}
        \centering
        \includegraphics[width=0.65\textwidth]{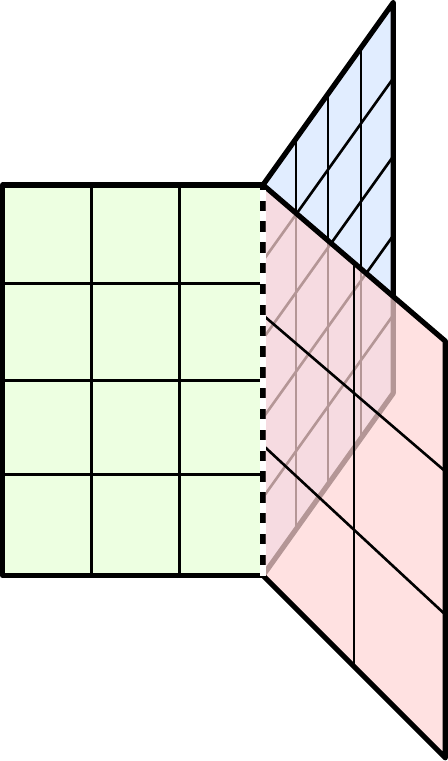}
        \caption{Non matching grids}
    \end{subfigure}
		
		\vspace{1em}
		\begin{subfigure}[b]{0.9\textwidth}
        \centering
        \includegraphics[width=0.65\textwidth]{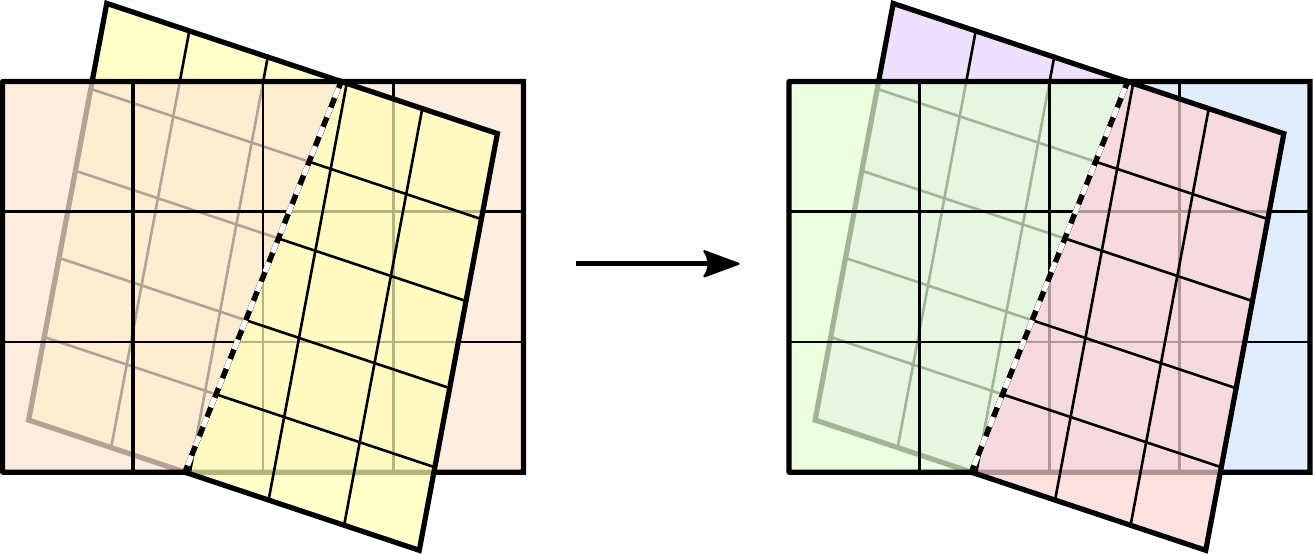}
        \caption{Intersection between meshed surfaces}
    \end{subfigure}
		
		\vspace{0.5em}
		\begin{subfigure}[b]{0.9\textwidth}
        \centering
        \includegraphics[width=0.56\textwidth]{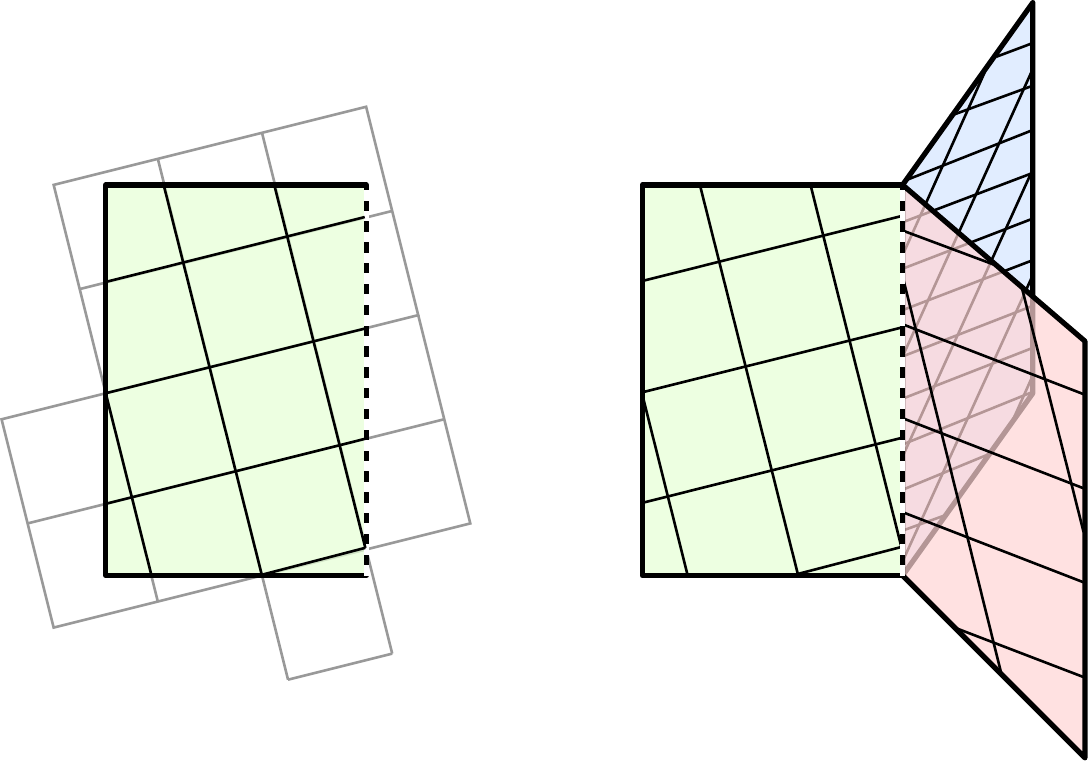}
        \caption{Surfaces with cut (unfitted) meshes}
    \end{subfigure}
		
      \caption{Illustrations of interface situations corresponding to the various mesh constructions for composite surfaces.}
    \label{fig:mesh-constructions}
\end{figure}

\paragraph{The Mesh.}
To accommodate these different situations we define the mesh as follows. 
For each $\Omega_{i}$ we assume that there is a family of quasi uniform 
meshes $\mcK_{h,i}$ with mesh parameter $h \in (0,h_0]$ such that 
\begin{equation}\label{eq:mesh-cover}
\overline{\Omega}_i \subseteq \cup_{K \in \mcK_{h,i}} \overline{K}
\end{equation}
and 
\begin{equation}\label{eq:mesh-intersection}
\text{int}(\Omega_i) \cap \text{int}(K) \neq \emptyset \qquad \forall K \in \mcK_{h,i}
\end{equation}
The mesh may match $\Omega$ perfectly, i.e, we could have equality in 
(\ref{eq:mesh-cover}).

\subsection{The Finite Element Spaces}

\paragraph{Finite Element Space.}
Let $P_p(K)$ denote the space of either full or tensor product polynomials of degree less or equal to $p$ on $K$.
For each mesh $\mcK_{h,i}$ there is a family of finite element spaces 
$V_{h,i}$, $h \in (0,h_0]$, such that $V_{h,i}|_K \in P_p(K)$ for all $K\in\mcK_{h,i}$. 
On the composite surface  $\Omega$ we 
define the broken finite element space
\begin{equation}
V_h = \bigoplus_{i \in \mcI_{\Omega}} V_{h,i}
\end{equation}

\paragraph{Approximation Property.} For each $i \in \mcI_\Omega$ there 
is an interpolation operator $\pi_{h,i}: L^2(\Omega_i) \rightarrow V_{h,i}$ 
such that
\begin{equation}\label{eq:interpol-general}
\|v - \pi_{h,i} v \|_{L^2(\mcK_{h,i})} 
+ h \|\nablas (v - \pi_{h,i} v) \|_{L^2(\mcK_{h,i})} 
\lesssim 
h^s \| v \|_{H^s(\Omega_i)}
\qquad 1 \leq s \leq p+1
\end{equation}
and on the composite surface we have the interpolation operator 
$\pi_h : L^2(\Omega) \rightarrow V_h$ such that $\pi_h |_{V_{h,i}} = \pi_{h,i}$ 
for $i \in \mcI_\Omega$. See \cite{JonLarLar17} for the construction of such an 
interpolation operator in the case of cut meshes.

\subsection{The Method}
\paragraph{Derivation.}
For $v \in V_h$ we obtain using Green's formula
\begin{align}
(f,v)_{\Omega} 
&= \sum_{ i\in \mcI_\Omega} (-\nablas \cdot \sigma (u_i), v_i)_{\Omega_i}
\\
&= \sum_{i \in \mcI_\Omega} (\sigma(u_i), \nablas v_i)_{\Omega_i} 
- (\nu_i \cdot \sigma(u_i),v_i)_{\partial\Omega_i}
\\
&= \sum_{i \in \mcI_\Omega} (\sigma(u_i), \nablas v_i)_{\Omega_i}
- (\nu \cdot \sigma(u),v)_{\Gamma_N}
- \sum_{j \in \mcI_{\Gamma_I}} \underbrace{\sum_{k \in \mcI_\Omega(j)}
(\nu_k \cdot \sigma(u_k),v_k)_{\Gamma_j}}_{\bigstar} 
\end{align}
Conservation \eqref{eq:prob-interface-flux} states that for each $j \in \mcI_{\Gamma_I}$,
\begin{equation}
\sum_{k \in \mcI_\Omega(j)} 
\nu_{k} \cdot \sigma(u_k) = 0  \qquad  \text{on $\Gamma_j$}
\end{equation}
which implies
\begin{equation}
\sum_{k \in \mcI_\Omega(j)} 
(\nu_{k} \cdot \sigma(u_k),\langle v \rangle )_{\Gamma_j} =
0
\end{equation}
where $\langle v \rangle$ is the convex combination 
\begin{equation}
\langle v \rangle = \sum_{k \in \mcI_{\Gamma_I}(j)} \alpha_k v_k
\end{equation}
with $0\leq \alpha_k \leq 1$ and $\sum_{k\in  \in \mcI_{\Gamma_I}(j)} 
\alpha_k = 1$. We may thus 
subtract $\langle v \rangle$ in the interface term $\bigstar$ 
as follows
\begin{align}
 \bigstar &= \sum_{k \in \mcI_\Omega(j)}
(\nu_k \cdot \sigma(u_k),v_k )_{\Gamma_j} 
=
 \sum_{k \in \mcI_\Omega(j)}
(\nu_k \cdot \sigma(u_k),v_k - \langle v \rangle)_{\Gamma_j }
\end{align}
Next for each $j \in \mcI_{\Gamma_I}$ we add the consistent 
stabilization term 
\begin{equation}
\sum_{k \in \mcI_{\Gamma_I}(j)} (\beta \mu_k  h^{-1} (u_k - \langle u \rangle ),v_k - \langle v \rangle)_{\Gamma_j} 
\end{equation}
where $\beta$ is a positive parameter.

We may also symmetrize the method by adding a 
consistent term and the Dirichlet boundary is taken care of 
using a standard Nitsche formulation. 

\paragraph{The Method.} Find $u_h \in V_h$ such that
\begin{equation}\label{eq:fem}
a_h(u_h,v) = l_h(v)\qquad \forall v \in V_h
\end{equation}
where 
\begin{align}
a_h(v,w)&=
 \sum_{i \in \mcI_\Omega} (\sigma(v), \nablas w)_{\Omega_i} 
 \\ \nonumber
 &\qquad
- \sum_{j \in \mcI_{\Gamma_I}} \sum_{k \in \mcI_\Omega(j)}
\left(
(\nu_k \cdot \sigma(v_k),w_k - \langle w \rangle)_{\Gamma_j} 
+
(v_k - \langle v \rangle,\nu_k \cdot \sigma(w_k))_{\Gamma_j} 
\right)
\\ \nonumber
 &\qquad
+ \sum_{j \in \mcI_{\Gamma_I}} \sum_{k \in \mcI_\Omega(j)}
 (\beta \mu_k h^{-1}(v_k - \langle v \rangle),w_k - \langle w \rangle)_{\Gamma_j} 
\\ \nonumber
 &\qquad
 - (\nu \cdot \sigma(v), w )_{\Gamma_D} 
 - ( v ,\nu \cdot \sigma(w))_{\Gamma_D}
+  (\beta \mu h^{-1} v ,  w  )_{\Gamma_D}
\end{align}
and
\begin{align}
l_h(w) &= (f,w)_\Omega + (g_N, w )_{\Gamma_N}
+
 \left(  g_D , \beta \mu h^{-1} w  - \nu \cdot \sigma(w) \right)_{\Gamma_D} 
\end{align}

%
%
%

\subsection{Relation to the Standard Average-Jump Formulation}

Consider an interface curve $\Gamma_j$ which has only two neighboring 
surfaces and assume that $\mu$ is a positive constant. We may assume 
that $\mcI_{\Gamma_I}(j) = \{1, 2\}$. Then we have 
 \begin{equation}
 \langle v \rangle = \alpha_1 v_1 + \alpha_2 v_2 
 \end{equation}
and the concistency term takes the form
\begin{align}\nonumber
&\sum_{k \in \mcI_\Omega(j)}
(\nu_k \cdot \sigma_k(v_k),w_k - \langle w \rangle)_{\Gamma_j} 
\\
&\qquad =
(\nu_1 \cdot \sigma_1(v_1),w_1 - \langle w \rangle)_{\Gamma_j}
+
(\nu_2 \cdot \sigma_2(v_2),w_2 - \langle w \rangle)_{\Gamma_j}
\\
&\qquad =
(\nu_1 \cdot \sigma_1(v_1),(1-\alpha_1)w_1 - \alpha_2 w_2 \rangle)_{\Gamma_j}
+
(\nu_2 \cdot \sigma_2(v_2),(1-\alpha_2)w_2 - \alpha_1 w_1 )_{\Gamma_j}
\\
&\qquad =
(\nu_1 \cdot \sigma_1(v_1),\alpha_2 w_1 - \alpha_2 w_2 \rangle)_{\Gamma_j}
+
(\nu_2 \cdot \sigma_2(v_2),\alpha_1 w_2 - \alpha_1 w_1 )_{\Gamma_j}
\\
&\qquad =
(\alpha_2 \nu_1 \cdot \sigma_1(v_1)-\alpha_1 \nu_2 \cdot \sigma_2(v_2),[w])_{\Gamma_j}
\\
&\qquad =
(\{\nu \cdot \sigma(v) \},[w])_{\Gamma_j}
\end{align}
where $[w] = w_1 - w_2$ is the jump in $w$ across $\Gamma_j$ and 
\begin{equation}
\{\nu \cdot \sigma(v) \} = \alpha_2 \nu_1 \cdot \sigma_1(v_1)-\alpha_1 \nu_2 \cdot \sigma_2(v_2)
\end{equation}
is the average of the normal flux. Note, that in the case when $\Omega_1$ and $\Omega_2$ are tangent at $\Gamma_j$ we have $\nu_2 = - \nu_1$ and we find that 
\begin{equation}
\{\nu \cdot \sigma(v) \} = \alpha_2 \nu_1 \cdot \sigma_1(v_1) + \alpha_1 \nu_1 \cdot \sigma_2(v_2)
\end{equation}
which is the usual average in discontinuous Galerkin methods. Note, in particular, 
that our formulation thus extends the standard discontinuous Galerkin formulations 
to surfaces with sharp edges. Finally, we easily find that the penalty term takes the form
\begin{align}\nonumber
&\sum_{k \in \mcI_\Omega(j)}
 (\beta \mu h^{-1}(v_k - \langle v_k \rangle),w_k - \langle w_k \rangle)_{\Gamma_j} 
\\
&\qquad = (\beta \mu h^{-1} (v_1 - \langle v \rangle ),w_1 - \langle w \rangle)_{\Gamma_j}
+ 
( \beta \mu h^{-1}(v_2 - \langle v \rangle),w_2 - \langle w \rangle)_{\Gamma_j}
\\
&\qquad = \beta \mu h^{-1}(\alpha_1^2 + \alpha_2^2)([ v ] ,[w] )_{\Gamma_j}
\end{align}
where we used the assumption that $\mu$ is constant and the identities 
$v_1 - \langle v \rangle = v_1 - (\alpha_1 v_1 + \alpha_2 v_2) = (1-\alpha_1) v_1 - \alpha_2 v_2 
= \alpha_2 [v]$ and $v_2 - \langle v \rangle = \alpha_1 [v ]$. We conclude that the penalty term
is of the same form as in the standard discontinuous Galerkin interior penalty term.

\subsection{The Stabilization Term}

\paragraph{Abstract Properties.}
In the case of cut elements at an interface we add a stabilization term 
of the form
\begin{equation}
s_h(v,w) = \sum_{i\in \mcI_\Omega} 
s_{h,i}(v,w)
\end{equation}
where $s_{h,i}$ is a positive semidefinite bilinear form on $V_{h,i}$, and define the stabilized 
form
\begin{equation}\label{def:Ah}
A_h(v,w)= a_h(v,w) + s_h(v,w)
\end{equation}
and the following seminorms on each subsurface $\Omega_i$
\begin{align}
\|v\|_{a_{i}}^2 = (\sigma_i(v),\nabla v)_{\Omega_i} \,,
\qquad
\|v\|_{s_{h,i}}^2 = s_{h,i}(v,v)
\end{align}
We assume that the stabilization term has the following properties: 
\begin{itemize}
\item There is a constant such that for all $i\in \mcI_\Omega$ and $v \in V_{h,i}$ it 
holds
\begin{alignat}{3}\label{eq:inverse-normal-flux}
h\| \nu_{i} \cdot \sigma_i (v ) \|^2_{\partial \Omega_i \setminus \partial \Omega_N} 
&\lesssim \| v \|^2_{a_{i}} +  \| v \|^2_{s_{h,i}} 
\end{alignat}
\item There is a constant such that for all $i \in \mcI_\Omega$ and 
$v\in H^{s+1}(\Omega_i)$ with $0 \leq s \leq p$ it holds 
\begin{equation}\label{eq:interpol-sh}
\| \pi_{h,i} v \|_{s_{h,i}} \lesssim h^s \| v \|_{H^{s+1}(\Omega_i)} 
\end{equation}
\end{itemize}
Here and below we use the abbreviated notation $a\lesssim b$ for the inequality $a\leq c b$ where $c$ is a constant independent of the mesh size parameter $h$.

\begin{rem} In order to prove a bound on the condition number of the stiffness matrix we 
assume that 
\begin{equation}\label{eq:assumption-condition-number}
\| v \|^2_{\mcK_{h,i}} \lesssim \| v \|^2_{\Omega_i} + \| v \|^2_{s_{h,i}}
\end{equation}
The condition number bound can then be proved using the techniques in \cite{JonLarLar17}.
\end{rem}

\begin{rem}Note that we do not assume that the stabilization form is consistent, i.e. for the exact solution 
$u$, $s_h(u,v) = 0$ for $v \in V_h$, even though this may be the case for a sufficiently regular 
solution.
\end{rem}
\paragraph{Normal Derivative Jump Penalty.} The stabilization form which we will use 
in this paper takes the form
\begin{equation}\label{eq:stab-F}
s_{h,i}(v,w) = \sum_{F \in \mcF_{h,i}} s_{h,i,F}(v,w),
\qquad
s_{h,i,F}(v,w) = \sum_{k=1}^p \gamma_{k,i} h^{2k-1} \left([ D^k_n v ] ,[ D^k_n w ] \right)_F
\end{equation}
where  $\mcF_{h,i}$ is the set of interior faces in the mesh $\mcK_{h,i}$, which belong to at 
least one element $K$ such that $\text{int}(K) \cap \partial \Omega_i \neq \emptyset$.  We note that 
$s_h$ is consistent for sufficiently regular functions, for instance for $v \in C^p(\Omega)$ we 
have $s_h(v,w) = 0$ for $w \in V_{h,i}$. 

We first note that $s_h$ is a bilinear positive semidefinite form by construction.  To verify (\ref{eq:inverse-normal-flux}) we recall, see \cite{JonLarLar17} Lemma 4.1, that for two neighboring 
elements $K_1$ and $K_2$ in $\mcK_{h,i}$ which share the face $F$ we have the estimate 
\begin{equation}
\|\nablas v \|^2_{K_1} \leq \| \nablas v \|^2_{K_2} +  \| v \|^2_{s_{h,i,F}}
\end{equation}   
and we may conclude that, for $h \in (0,h_0]$ with $h_0$ small enough,  
\begin{equation}\label{eq:technical-stab-aa}
\| \nablas v \|^2_{\mcK_{h,i}} \lesssim  \| \nablas v \|^2_{\text{int}(\mcK_{h,i})} 
+ \| v \|^2_{s_{h,i}}
\leq 
\| v \|^2_{a_i}  + \| v \|^2_{s_{h,i}}
\end{equation}
where $\text{int}(\mcK_{h,i}) = \{ K \in \mcK_{h,i}: \text{int}(K) \cap \partial \Omega_{i} = \emptyset \}$ 
is the set of elements that do not intersect the boundary. Finally, we have the inverse estimate 
\begin{equation}\label{eq:trace-ineq-aa}
\| \nu_{h,i} \cdot \sigma(v) \|_{\partial \Omega_{i}} 
\lesssim 
\| \nabla v \|_{\mcK_{h,i}(\partial \Omega_i)}  
\end{equation}
where $\mcK_{h,i}(\partial \Omega_i)$ is the set of elements that intersect the boundary 
in a face or in the interior. See \cite{HaHaLa03} for the trace inequality  
$\| v \|^2_{\partial \Omega_i\cap K} \lesssim h^{-1} \| v \|^2_K + h \| \nabla v \|^2_K$, $v \in H^1(K)$, 
from which (\ref{eq:trace-ineq-aa}) follows. Combining estimates (\ref{eq:technical-stab-aa}) and 
(\ref{eq:trace-ineq-aa}) we find that 
\begin{equation}
h\| \nu_{h,i} \cdot \sigma(v) \|^2_{\partial \Omega_{i}}
\lesssim 
\| \nabla v \|^2_{\mcK_{h,i}(\partial \Omega_i)}  
\lesssim 
\| \nablas v \|^2_{\mcK_{h,i}} 
\lesssim 
\| v \|^2_{a_i}  + \| v \|^2_{s_{h,i}}
\end{equation}
To verify (\ref{eq:interpol-sh}) we consider again the pair of two elements $K_1$ and $K_2$ 
sharing the face $F$ and we let $w \in P_p(K_1 \cup K_2 )$. We note that 
$s_{h,i,F}(w,w) = 0$ and thus 
\begin{align}
\| \pi_{h,i} v \|^2_{s_{h,i,F}} &= \| \pi_{h,i} v - w \|^2_{s_{h,i,F}} 
\\
&\lesssim \sum_{j=1}^2 \| \nabla (\pi_{h,i} v - w ) \|^2_{K_j}
\\
&\lesssim  \sum_{j=1}^2 \| \nabla (\pi_{h,i} v - v) \|^2_{K_j} + \|\nabla( v - w)  \|^2_{K_j}
\\
&\lesssim  \sum_{j=1}^2 \| \nabla (\pi_{h,i} v - v) \|^2_{K_j}  + h^{2s} \| v \|^2_{H^{s+1}(K_j)}
\end{align}
where we used an inverse inequality and the Bramble--Hilbert Lemma to estimate 
the second term. Summing over all $F\in \mcF_{h,i}$ and using  the approximation 
property (\ref{eq:interpol-general}) give (\ref{eq:interpol-sh}).

\paragraph{Least Squares Gradient Variation Penalty.} Define 
\begin{equation}\label{eq:stab-F-ver2}
s^1_{h,i}(v,w) = \sum_{F \in \mcF_{h,i}} s^1_{h,i,F}(v,w),
\qquad
s^1_{h,i,F}(v,w) = ( \nabla (v - P_F v), \nabla (w - P_F w)_{K_1\cup K_2}
\end{equation}
where $P_F: H^1(K_1 \cup K_2) \rightarrow P_p(K_1\cup K_2 )$ is the $H^1$ projection. This 
stabilization term is not consistent but it is not difficult to verify that it satisfies the conditions (\ref{eq:interpol-sh}) and (\ref{eq:interpol-sh}) using similar arguments as above. See also \cite{Bur10} 
where related stabilization terms were studied.

\section{Properties of the Finite Element Method}\label{sec:analysis}

\subsection{Coercivity and Continuity}

We define the norms 
\begin{align}
\tn v \tn_{a_h}^2 &= 
\sum_{i \in \mcI_\Omega} \| v \|^2_{a_{i}} 
+ 
\sum_{j \in \mcI_{\Gamma_I}} 
\sum_{k \in \mcI_\Omega(j)} \Big( h \| \nu_k \cdot \nablas v_k \|^2_{\Gamma_j}
+h^{-1} \| v_k - \langle v \rangle \|^2_{\Gamma_j} \Big)
\\ \nonumber
&\qquad + h \| \nu \cdot \nablas v \|^2_{\Gamma_D}
+h^{-1} \| v \|^2_{\Gamma_D}
\end{align}
and 
\begin{align}
\tn v \tn_{A_h}^2 = \tn v \tn_{a_h}^2 + \| v \|^2_{s_h}
\end{align}
where
\begin{align}
\|v\|_{a_{i}}^2 = (\sigma_i(v),\nabla v)_{\Omega_i} \,,
\qquad
\|v\|_{s_{h}}^2 = s_h(v,v)
\end{align}

\begin{lem} For large enough penalty parameter $\beta$, there is a constant such that for 
all $v\in V_h$ it holds
\begin{equation}\label{eq:coercivity}
\tn v \tn_{A_h}^2 \lesssim A_h(v,v) 
\end{equation}
There is a constant such that for $v,w \in V_h + V$ it holds
\begin{equation}
a_h(v,w) \lesssim \tn v \tn_{a_h} \tn w \tn_{a_h}
\end{equation}
and
\begin{equation}\label{eq:continuity}
A_h(v,w) \lesssim \tn v \tn_{A_h} \tn w \tn_{A_h}
\end{equation}
\end{lem}
\begin{proof} For large enough penalty parameter $\beta$ the coercivity follows 
from the fact that the inverse inequality (\ref{eq:inverse-normal-flux}) holds 
together with standard arguments.
\end{proof}

\subsection{Interpolation Error Estimate}
We have the estimate 
\begin{equation}\label{eq:interpol-energy}
\tn v - \pi_h v \tn^2_{a_h} 
\lesssim 
\sum_{i \in \mcI_\Omega}
h^{2(s_i-1)} \| v \|^2_{H^{s_i}(\Omega)}, \qquad 3/2 < s_i \leq p+1
\end{equation}
This inequality follows from the trace inequality 
\begin{equation}
\| w \|^2_{\Gamma_j \cap K} 
\lesssim h^{-1} \| w \|_K^2 + h \|\nablas w \|^2_K  
\end{equation}
for each element $K \in \mcK_{h,i}$, $i \in \mcI_{\Omega}(j)$ with a nonempty 
intersection with $\Gamma_j$ and the interpolation estimate (\ref{eq:interpol-general}).

\subsection{Error Estimates} 

In the proof of the $L^2$ error estimate we will use a duality argument and we now 
specify the regularity properties needed for our analysis. Let $\phi \in V_0$ 
be the solution to the dual problem
\begin{equation}\label{eq:dual-problem}
a(v,\phi) = (\psi,v)_\Omega\qquad v \in V_0
\end{equation}
where $\psi \in L^2(\Omega)$. We assume that there are constants 
$\eta_i^* \in (3/2,2], i \in \mcI_\Omega$, and a hidden constant such that for all 
$\psi \in L^2(\Omega)$, the regularity estimate 
\begin{equation}\label{eq:elliptic-regularity-dual}
\sum_{i \in \mcI_\Omega} \| \phi \|_{H^{\eta_i^*}(\Omega)}^2 \lesssim \| \psi \|_\Omega^2 
\end{equation}
holds.

\begin{thm}[Error Estimates] Assume that the exact solution $u$ to 
(\ref{eq:exactsol}) satisfies the regularity estimate (\ref{eq:regularity}), then there is a constant 
such that 
\begin{equation}\label{eq:error-estimate-energy}
\tn u - u_h \tn_{a_h}^2  + \| \pi_h u - u_h \|^2_{s_h} 
\lesssim \sum_{i\in \mcI_\Omega} 
h^{2(\widetilde{\eta_i} - 1) } \| u \|^2_{H^{\eta_i}(\Omega_i)}
\end{equation}
with $\widetilde{\eta}_i = \min(\eta_i, p+1)$

If in addition the solution 
to the dual problem (\ref{eq:dual-problem}) satisfies the regularity estimate 
(\ref{eq:elliptic-regularity-dual}), then  there is a constant such that
\begin{equation}\label{eq:error-estimate-L2}
\| u - u_h \|^2_\Omega  
\lesssim \sum_{i\in \mcI_\Omega} 
h^{2 (\widetilde{\eta}_i + \eta^* - 2) } \| u \|^2_{H^{\eta_i}(\Omega_i)}
\end{equation}
where $\eta^* = \min_{i\in\mcI_\Omega} \eta^*_i$
\end{thm}
\begin{rem} Note that it follows from (\ref{eq:error-estimate-energy}) and (\ref{eq:interpol-sh}) that 
\begin{equation}
\| u_h \|^2_{s_h} 
\lesssim 
\|\pi_h u \|^2_{s_h} + \| \pi_h u - u_h \|^2_{s_h} 
\lesssim 
h^{2(\eta_i - 1) } \| u \|^2_{H^{\eta_i}(\Omega_i)}
\end{equation}

\end{rem} 
\begin{proof} {\bf(\ref{eq:error-estimate-energy}).} Splitting the error into an interpolation error and a discrete error we 
have 
\begin{align}
\tn u - u_h \tn^2_{a_h} + \| \pi_h u - u_h \|^2_{s_h}  
&\lesssim
\tn u - \pi_h u \tn^2_{a_h} 
+\underbrace{\tn \pi_h u - u_h \tn^2_{a_h}
+  \tn \pi_h u - u_h \tn^2_{s_h}}_{\tn \pi_h u - u_h \tn_{A_h}^2}
\\ \label{eq:error-split}
&\lesssim 
\sum_{i \in \mcI_\Omega} h^{2(\eta_i-1)} 
\| u \|^2_{H^{\eta_i}(\Omega)}
+ \tn \pi_h u - u_h \tn_{A_h}^2
\end{align}
where we used the energy norm interpolation error estimate (\ref{eq:interpol-energy}). 
Next using coercivity (\ref{eq:coercivity})  we have
\begin{align}\label{eq:inf-sup}
\tn \pi_h u - u_h \tn_{A_h} \lesssim \sup_{v \in V_h \setminus \{0\}} \frac{A_h( \pi_h u - u_h, v)}{\tn v \tn_{A_h} }
\end{align}
Here the numerator may rewritten, using the definition of the method (\ref{eq:fem}),
$A_h(u_h,v)  = l_h(v)$, $v \in V_h$,  and the fact that the unstabilized method is 
consistent  $a_h(u,v) = l_h(v)$, $v \in V_h$, as follows
\begin{align}
A_h( \pi_h u - u_h, v ) &= A_h(\pi_h u , v ) - l_h(v)
\\
&=a_h(\pi_h u , v ) + s_h(\pi_h u , v ) - l_h(v)
\\
&=a_h(\pi_h u - u , v ) + l_h(v ) + s_h(\pi_h u , v ) - l_h(v)
\\
&=a_h(\pi_h u - u , v )  +  s_h(\pi_h u , v )
\end{align}
Thus we find that we have the bound
\begin{align}
A_h( \pi_h u - u_h, v )
&\lesssim 
\tn  \pi_h u - u \tn_{a_h}  \tn v \tn_{a_h} 
+ \| \pi_h u \|_{s_h} \| v \|_{s_h} 
\\
&\lesssim 
( \tn  \pi_h u - u \tn^2_{a_h} + \| \pi_h u \|^2_{s_h} )^{1/2}   \tn v \tn_{A_h} 
\\ \label{eq:denominator-bound}
&\lesssim 
\left( \sum_{i \in \mcI_\Omega} h^{2(\eta_i-1)} \| u \|^2_{H^{\eta_i}(\Omega)} \right)^{1/2}
\tn v \tn_{A_h}
\end{align}
where we used the energy norm interpolation estimate (\ref{eq:interpol-energy}) and 
the approximation property (\ref{eq:interpol-sh}) of $s_h$.  
Combining (\ref{eq:inf-sup}) and (\ref{eq:denominator-bound}),  gives
\begin{equation}
\tn \pi_h u - u_h \tn_h^2 \lesssim \sum_{i \in \mcI_\Omega} h^{2(\eta_i-1)} 
\| u \|^2_{H^{\eta_i}(\Omega)}
\end{equation}
which together with (\ref{eq:error-split}) concludes the proof.

\paragraph{(\ref{eq:error-estimate-L2}).} 
Let $\phi$ be the solution to the dual problem (\ref{eq:dual-problem}) with 
$\psi = e$, where $e = u - u_h$ is the error. We note that by consistency we 
have 
\begin{equation}
a_h(v,\phi) = (e,v)_\Omega \qquad \forall v \in V
\end{equation}
Setting $v=e$ we obtain
\begin{align}
\| e \|_\Omega^2 &= a_h(e, \phi) 
\\
&=a_h(e, \phi - \pi_h \phi ) + a_h(e, \pi_h \phi ) 
\\
&=a_h(e, \phi - \pi_h \phi ) + a_h(u , \pi_h \phi ) -  a_h( u_h, \pi_h \phi ) 
\\
&=a_h(e, \phi - \pi_h \phi ) + 
\underbrace{l_h(\pi_h \phi ) 
-  a_h( u_h, \pi_h \phi ) - s_h(u_h,\pi_h  \phi )}_{=0\;\; (\ref{eq:fem})}+  s_h(u_h,\pi_h  \phi )
\\
&= 
a_h(e, \phi - \pi_h \phi ) + s_h(u_h - \pi_h u ,\pi_h  \phi ) + s_h( \pi_h u ,\pi_h  \phi )
\\
&\lesssim 
\tn e \tn_{a_h} \tn \phi - \pi_h \phi \tn_{a_h} 
+ \| u_h - \pi_h u \|_{s_h} \| \pi_h \phi \|_{s_h} 
 + \| \pi_u u \|_{s_h} \| \pi_h \phi \|_{s_h}
\\
&\lesssim 
\left( \sum_{i \in \mcI_\Omega} h^{2(\eta_i-1)} \| u \|^2_{H^{\eta_i}(\Omega)} \right)^{1/2}
\left(
\sum_{i \in \mcI_\Omega} h^{2(\eta_i^*-1)}
 \| \phi \|_{H^{\eta_i^*}(\Omega_i)}^2
 \right)^{1/2}
\\
&\lesssim 
\left( \sum_{i \in \mcI_\Omega} h^{2(\eta_i-1)} \| u \|^2_{H^{\eta_i}(\Omega)} \right)^{1/2}
h^{\eta^*-1}
 \| e \|_\Omega
\end{align}
Here we used the energy norm bound (\ref{eq:error-estimate-energy}) and the 
energy interpolation estimate (\ref{eq:interpol-energy}) followed by the elliptic 
regularity bound (\ref{eq:elliptic-regularity-dual}) to conclude that
the following estimate holds
\begin{equation}
\tn \phi - \pi_h \phi \tn \lesssim h^{\eta^*-1} \| \phi \|_{H^{\eta^*}(\Omega)} \lesssim \| e \|_\Omega
\end{equation}
\end{proof}

\newcommand{\Nitsche}{\gamma_1}
\newcommand{\Stab}{\gamma_2}

\section{Numerical Examples}\label{sec:numerical}
\paragraph{Geometries.}
We demonstrate the method using the three composite surfaces illustrated in Figure~\ref{fig:geom}, where each geometry has different features:
\begin{itemize}
	\item The \emph{cube with holes} in Figure~\ref{fig:cube-domain} consists of 6 separate surfaces pairwise connected along 12 interface lines. The resulting compound surface feature both sharp edges and corners.
	\item The \emph{house of cards} in Figure~\ref{fig:cards-domain} consists of 18 separate surfaces and 10 interface lines. Here each interface connects 2-6 surfaces. Also this geometry has sharp edges.
	\item The \emph{intersecting cylinders} composite surface in Figure~\ref{fig:cylinders-domain} is constructed by intersecting two cylinder surfaces and cutting the surfaces along their intersection. This construction produces 6 surfaces connected along the interface curves. In this geometry both the surfaces and the interfaces are curved and 4 surfaces meet at each interface.
\end{itemize}

\paragraph{General Construction.}
All examples share the following set-up:
\begin{itemize}
\item The geometry of each surface is exactly described by a mapping $F \colon \IR^2 \supset \widehat{\Omega} \to \Omega$, where $\widehat{\Omega}_i$ is the reference domain and $F(\widehat{\Omega}_i) = \Omega_i$.  
\item The elements used in all examples are parametric quadratic tensor product Lagrange elements ($p=2$)
and we allow cut elements at the interfaces. The parametrization is based on the exact map $F$, 
\begin{equation}
V_h = \widehat{V}_h\circ{F^{-1}}
\end{equation}
where $\widehat{V}_h$ is the finite element space in the reference domain. The stabilization term $s_h$ 
is evaluated in the reference domain, see \cite{JonLarLar17} for further details.
\item 
Our Nitsche penalty parameter is chosen to be $\beta=100$ and our stabilization parameter to be $\gamma_k=10^{-2}$.
\item 
We solve the model problem $-\nabla_{\Omega} \cdot \sigma(u) = f$, where $\sigma(u) = \mu \nabla_\Omega u$ and $\mu=1$, together with Dirichlet and Neumann boundary 
conditions.
\end{itemize}

\begin{figure}
\centering
  \begin{subfigure}[b]{0.4\textwidth}\centering
    \includegraphics[width=0.9\textwidth]{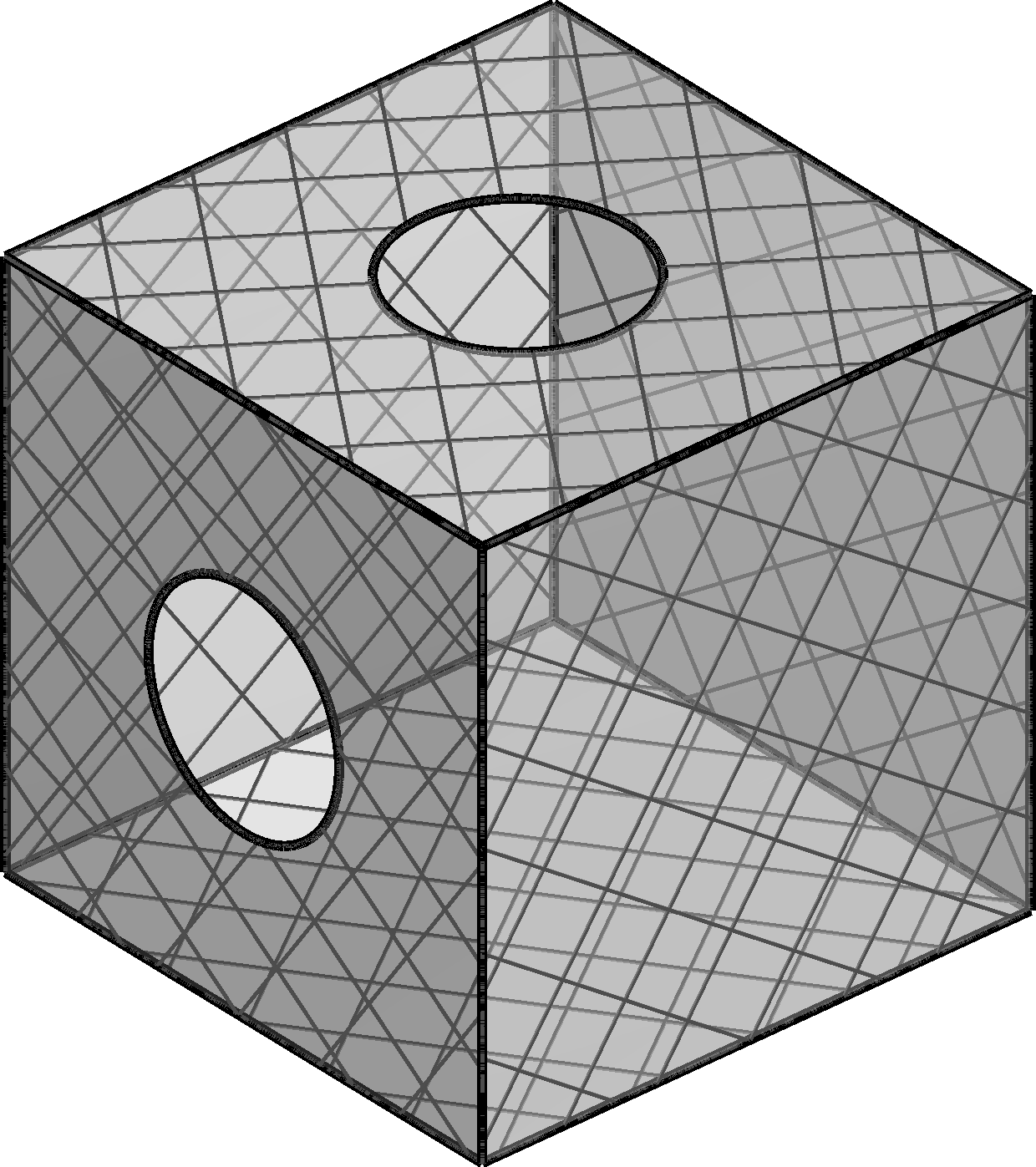}
    \caption{Cube with holes}
    \label{fig:cube-domain}
  \end{subfigure}
  \begin{subfigure}[b]{0.4\textwidth}\centering
    \includegraphics[width=0.9\textwidth]{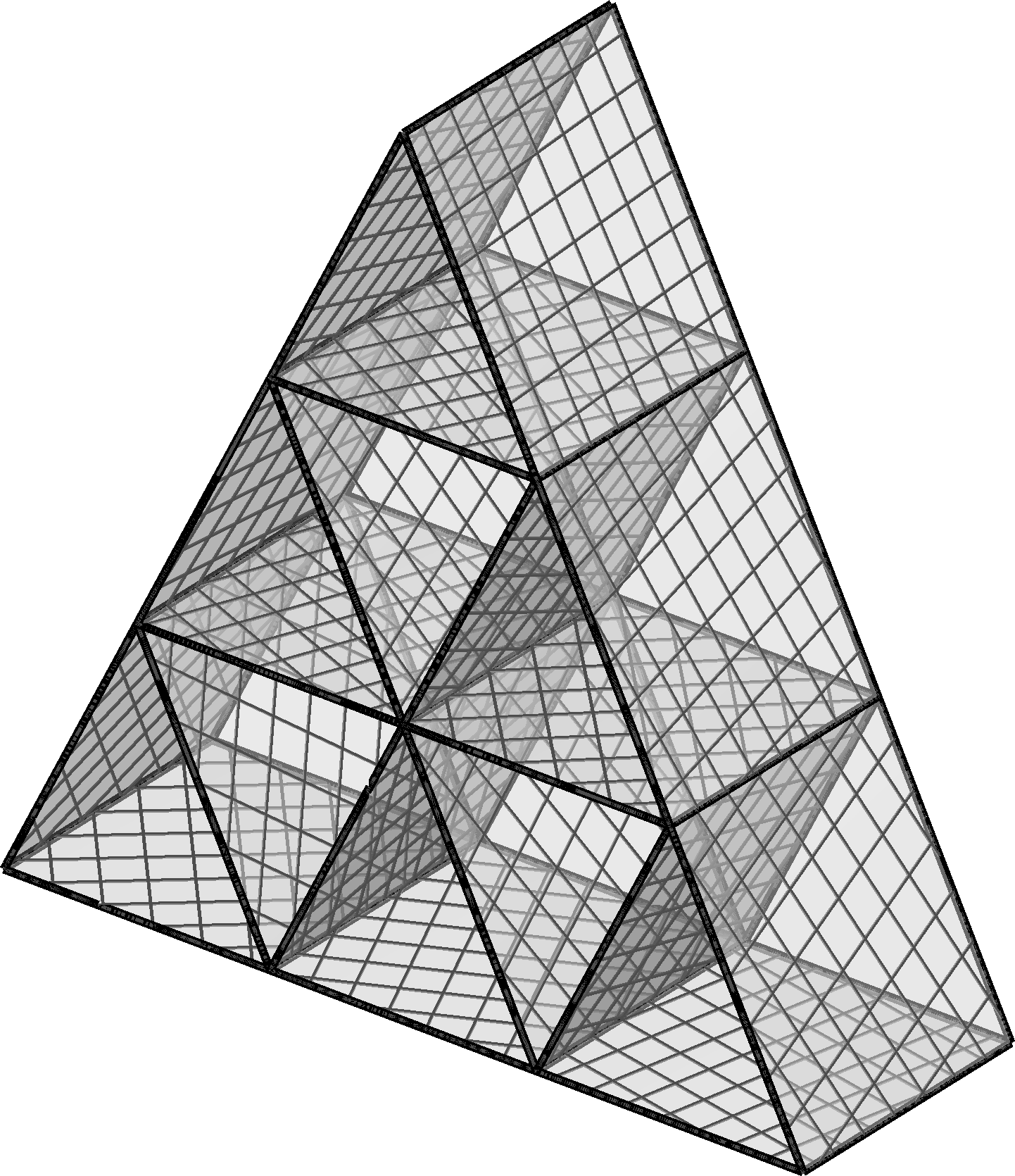}
   \caption{House of cards}
    \label{fig:cards-domain}
  \end{subfigure}

\vspace{2em}

  \begin{subfigure}[b]{\textwidth}\centering
    \includegraphics[width=0.55\textwidth]{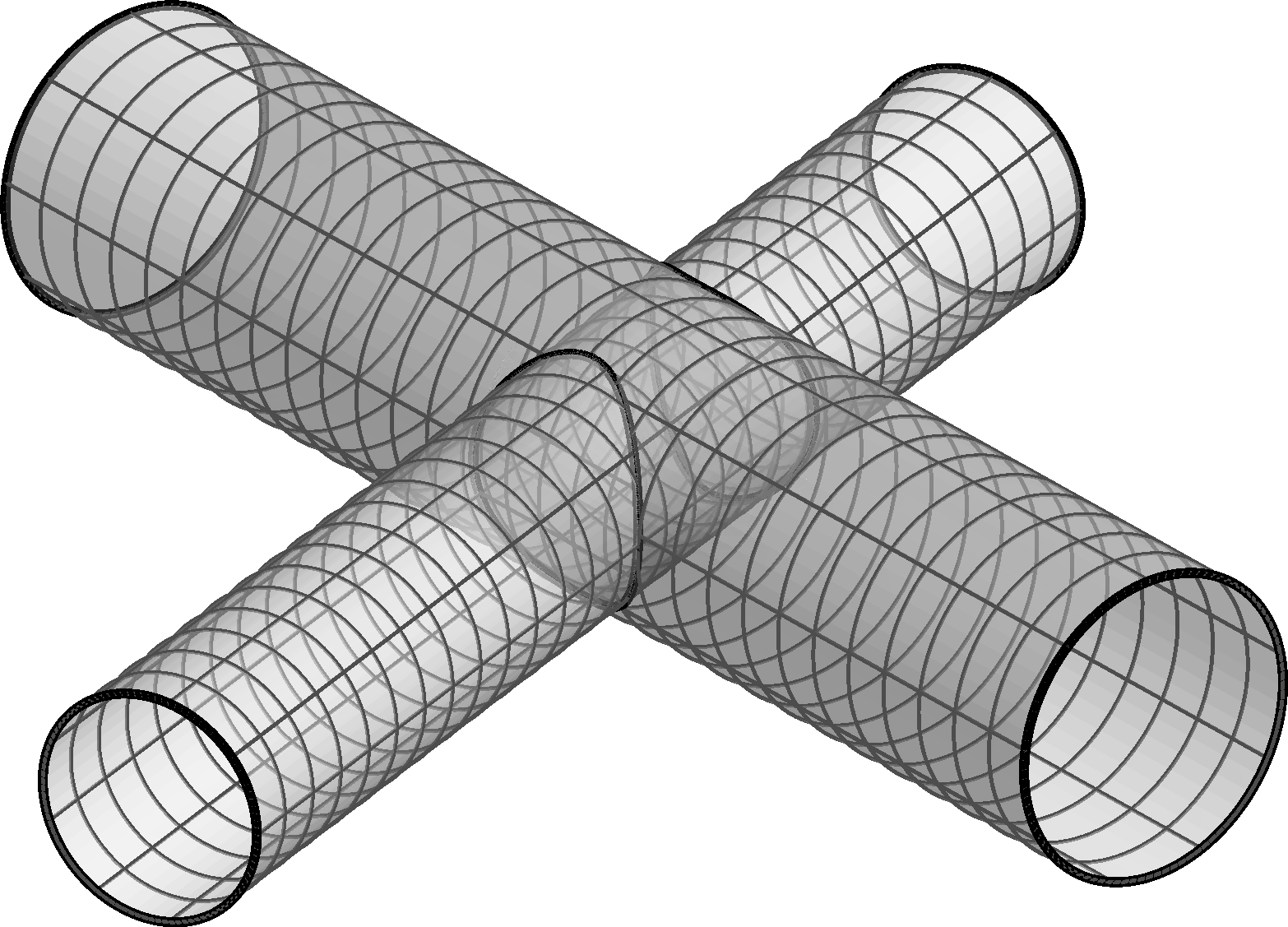}
   \caption{Intersecting cylinders}
    \label{fig:cylinders-domain}
  \end{subfigure}
  \caption{The composite surface geometries used in the numerical examples, here presented with a structured unfitted mesh on each surface. }
  \label{fig:geom}
\end{figure}

\subsection{Visual Examples}
\paragraph{Cube with holes.}
The cube illustrated in Figure~\ref{fig:cube-domain} is an example of a composite surface which includes both sharp edges and corners. We consider our model problem with load $f=0$ and non-homogeneous Dirichlet boundary conditions. The boundary in this case consist of the two holes and we let $u=1$ along the top hole and $u=-1$ along the side hole. We present a finite element solution and the magnitude of its gradient in Figure~\ref{fig:cubesol} using unfitted meshes on each surface. Note that as implied by the interface conditions \eqref{eq:prob-interface-flux}--\eqref{eq:prob-interface-cont} both the solution and the flux seem to flow continuously over the interfaces. Curiously, the gradient near the cube corners do not seem to include any singularity, as could possibly be expected, but is rather zero 
at the corner. We highlight this in Figure~\ref{fig:cubecorner}.
\begin{figure}
\centering
  \begin{subfigure}[b]{0.45\textwidth}
    \includegraphics[width=0.9\textwidth]{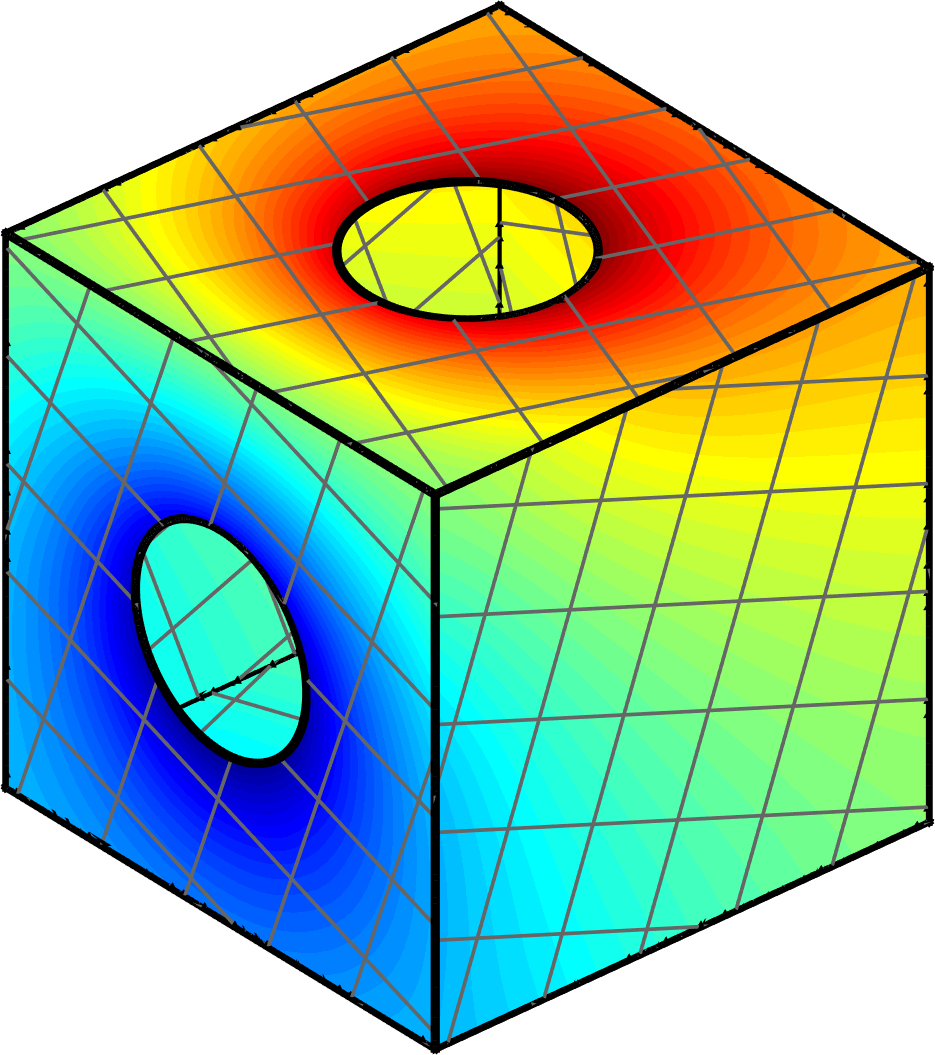}
    \caption{Solution}
  \end{subfigure}
  \begin{subfigure}[b]{0.45\textwidth}
    \includegraphics[width=0.9\textwidth]{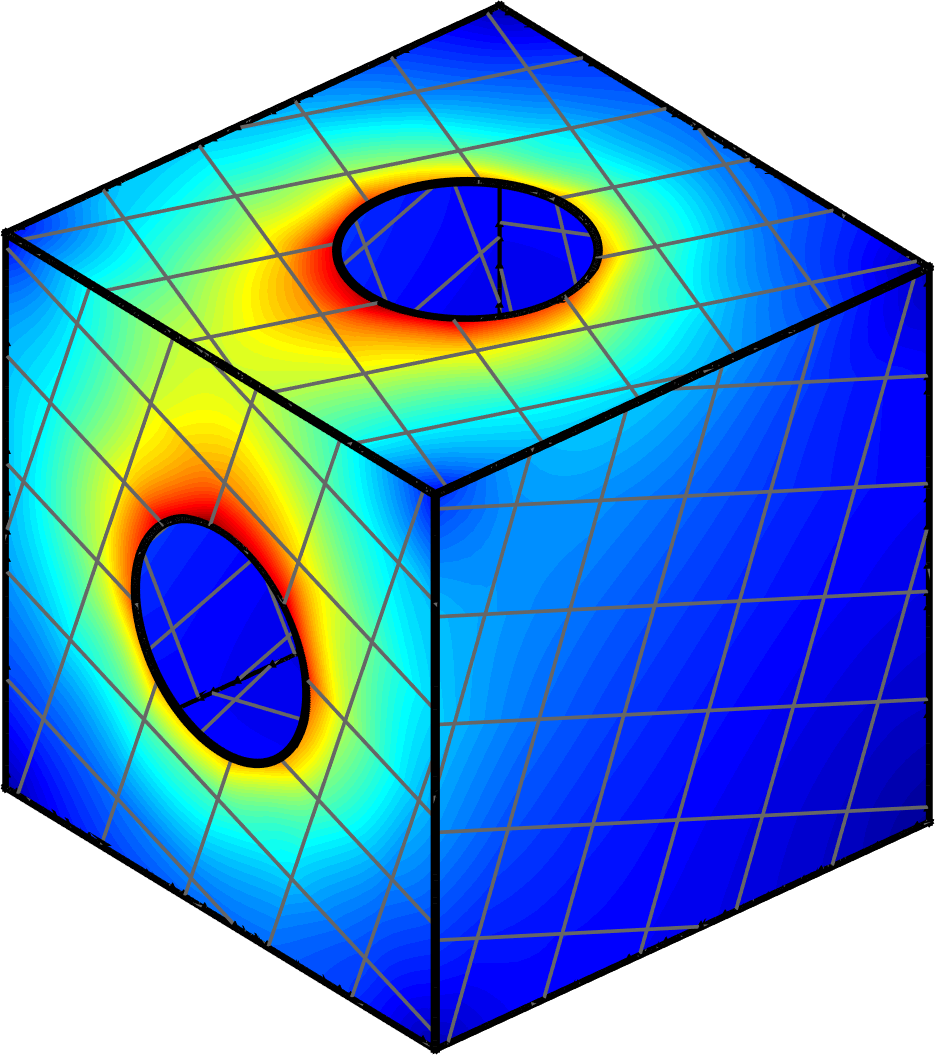}
   \caption{Gradient magnitude}
  \end{subfigure}
  \caption{Finite element solution to a cube with Dirichlet boundary conditions on the holes. Note here that the magnitude of the gradient is zero at the corners. }
  \label{fig:cubesol}
\end{figure}

\begin{figure}
\centering
    \includegraphics[width=0.45\textwidth]{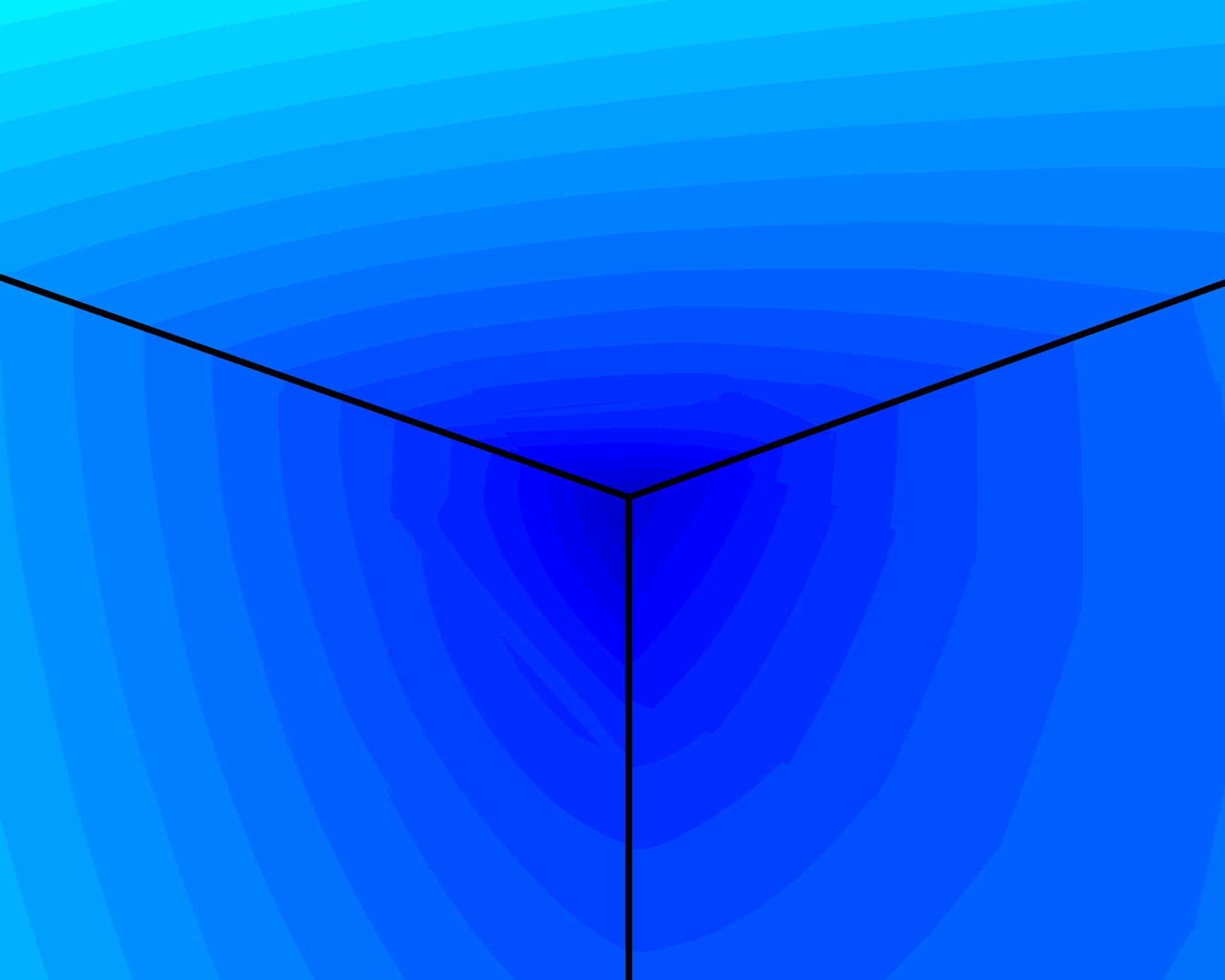}
   \caption{Detail of the gradient magnitude at the cube corner.}
  \label{fig:cubecorner}
\end{figure}

\paragraph{House of Cards.} The house of cards composite surface shown in Figure~\ref{fig:cards-domain} includes interfaces joining more than two surfaces. Again we consider our model problem with $f=0$ and we set the non-homogeneous Dirichlet boundary condition $u=1$ on the right side of the top card closest in view and $u=-1$ on the left side of the bottom standing card closest in view. On the remainder of the boundary we use homogeneous Neumann conditions. This leads to singularities in the points where we switch boundary conditions. Note that the 
singularities spread to the surfaces that couple at the interface. However, as the Kirchhoff condition \eqref{eq:prob-interface-flux} states that the sum of the conormal fluxes should be zero we no longer have pairwise continuity of fluxes over interfaces joining more than two surfaces. We can see this in Figure~\ref{fig:cardzoom} where we also note the peculiar effect that the singularities due to the switch of boundary condition at the interfaces has a greater effect on interfaces joining more surfaces. While somewhat counter intuitive we isolate this effect in the study presented in Figure~\ref{fig:singularity-study}.

\begin{figure}
\centering
  \begin{subfigure}[b]{0.45\textwidth}
    \includegraphics[width=0.9\textwidth]{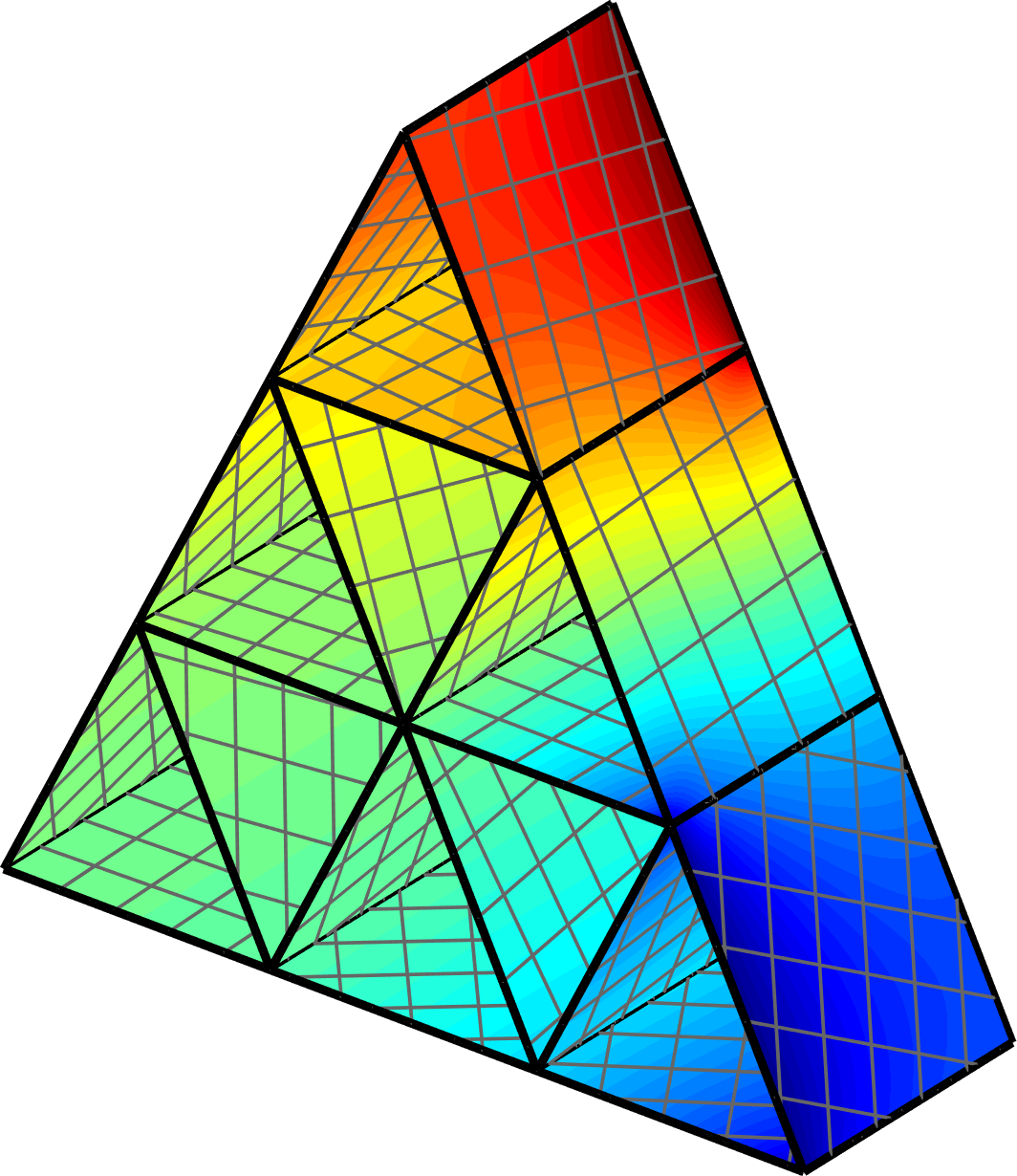}
   \caption{Solution}
  \end{subfigure}
  \begin{subfigure}[b]{0.45\textwidth}
    \includegraphics[width=0.9\textwidth]{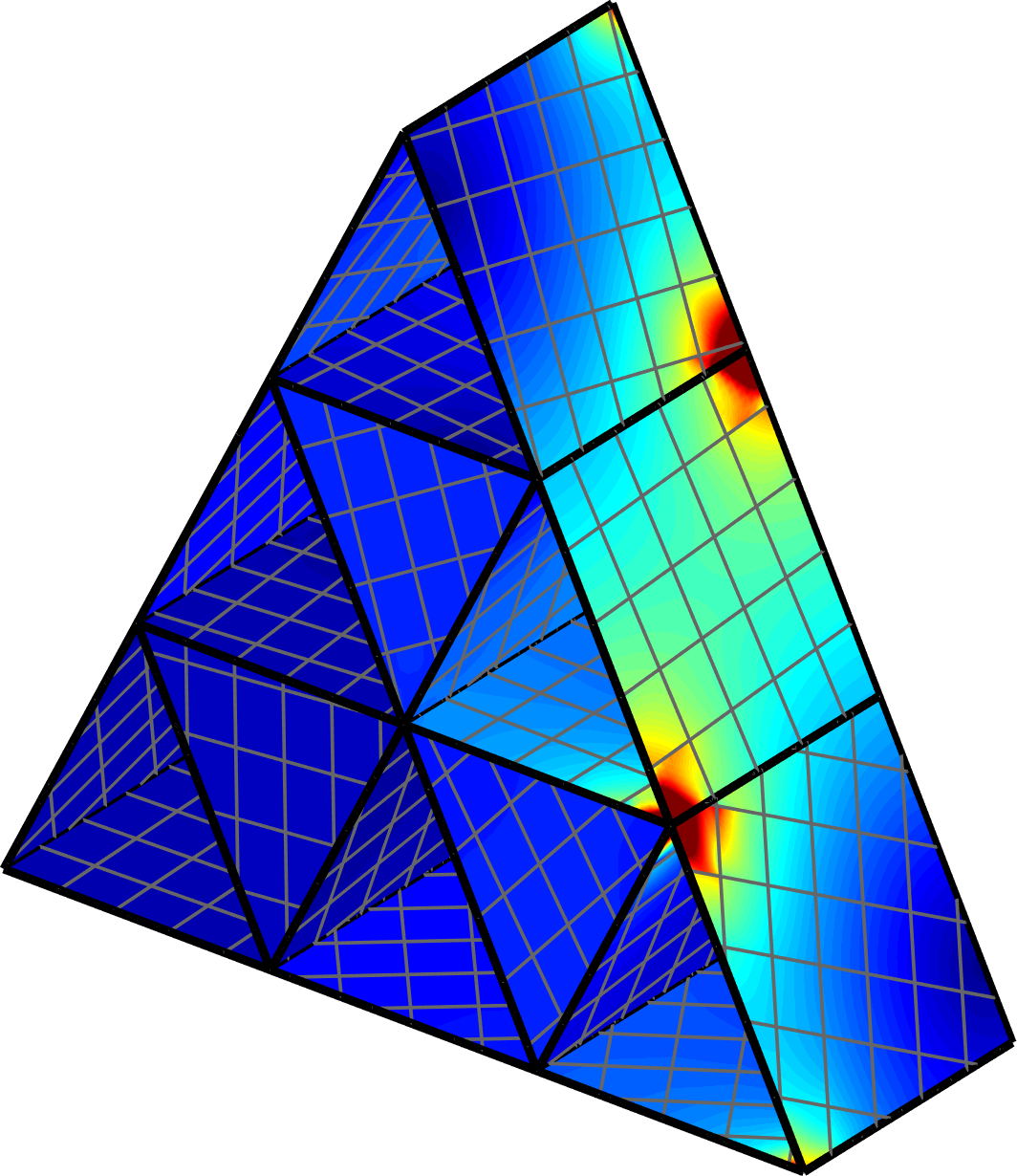}
  \caption{Magnitude of the gradient}
  \end{subfigure}
   \caption{Finite element solution to the house of cards cube with Dirichlet boundary conditions on two card edges. Over the interface the gradient does not have to be continuous due to the interface condition. }
  \label{fig:cardsol}
\end{figure}

\begin{figure}
\centering
\includegraphics[width=0.45\textwidth]{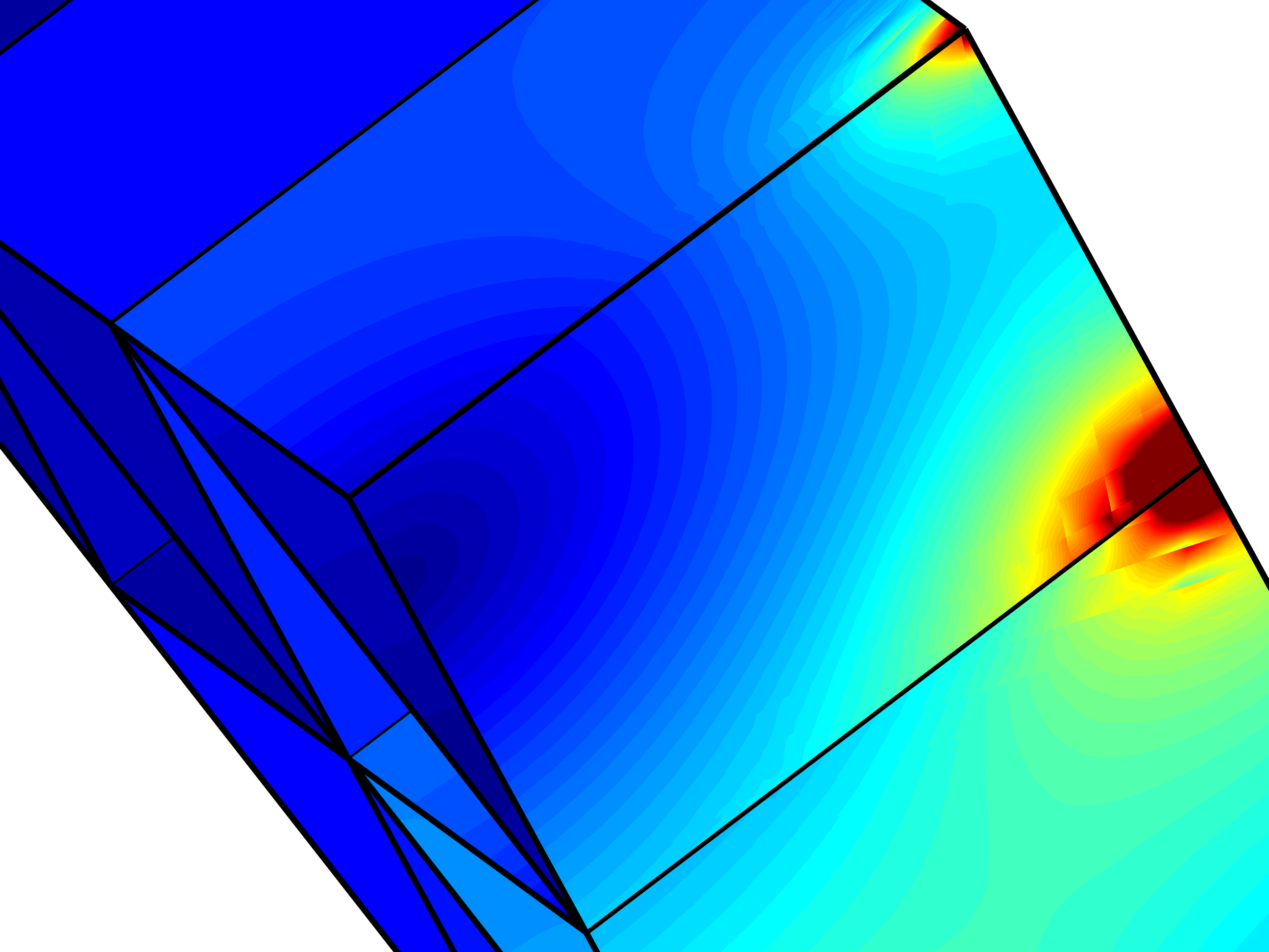}
\caption{A close-up on the top of the house of cards, where only two surfaces meet and the gradient flows continuously over the interface. The interface beneath consist of more than two surfaces, and the gradient does not flow continuously over the interface.  }
\label{fig:cardzoom}
\end{figure}

\begin{figure}
\centering
  \begin{subfigure}[t]{0.24\linewidth}\centering
    \includegraphics[height=3.5cm]{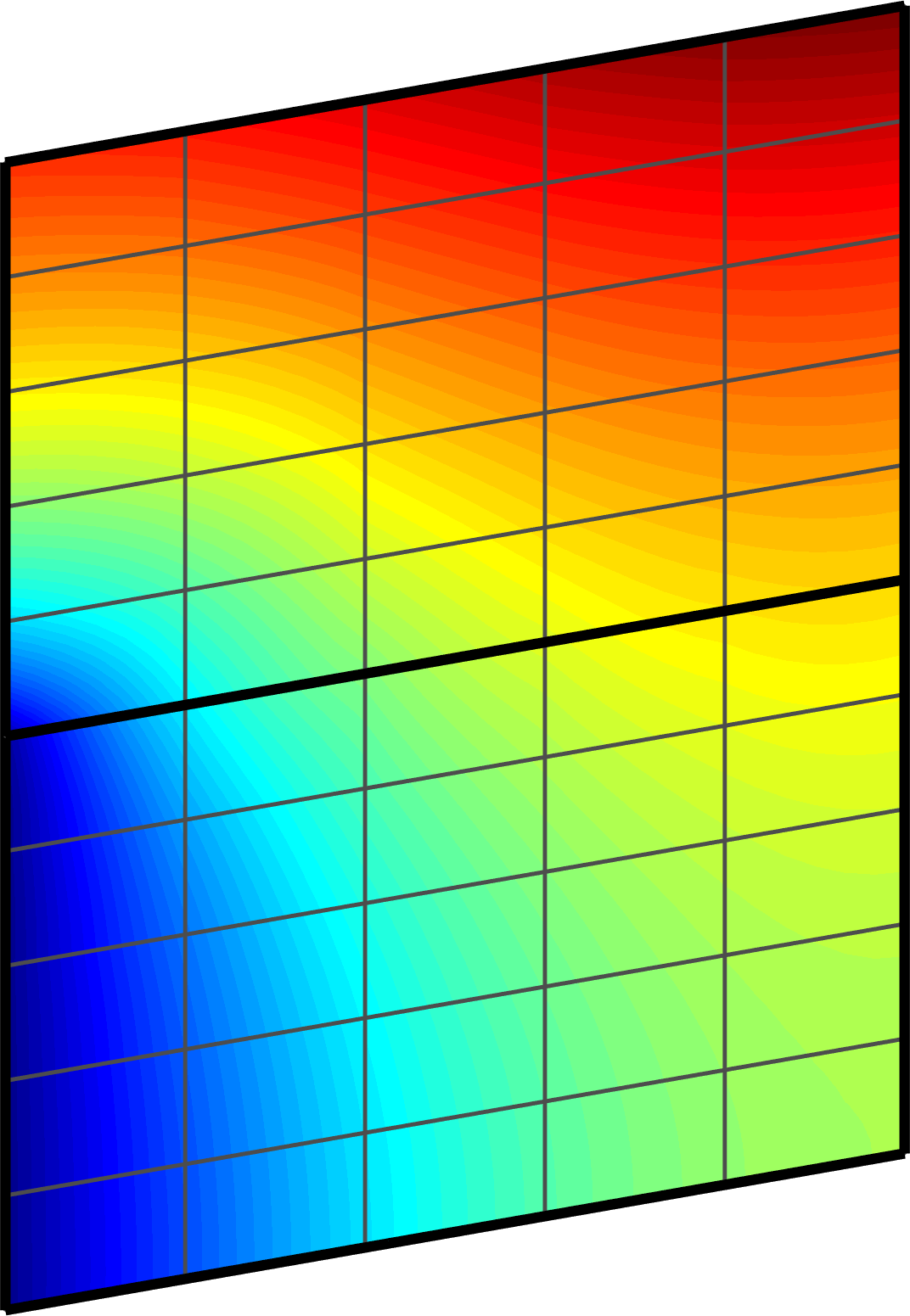}
  \end{subfigure}
  \begin{subfigure}[t]{0.24\linewidth}\centering
    \includegraphics[height=3.5cm]{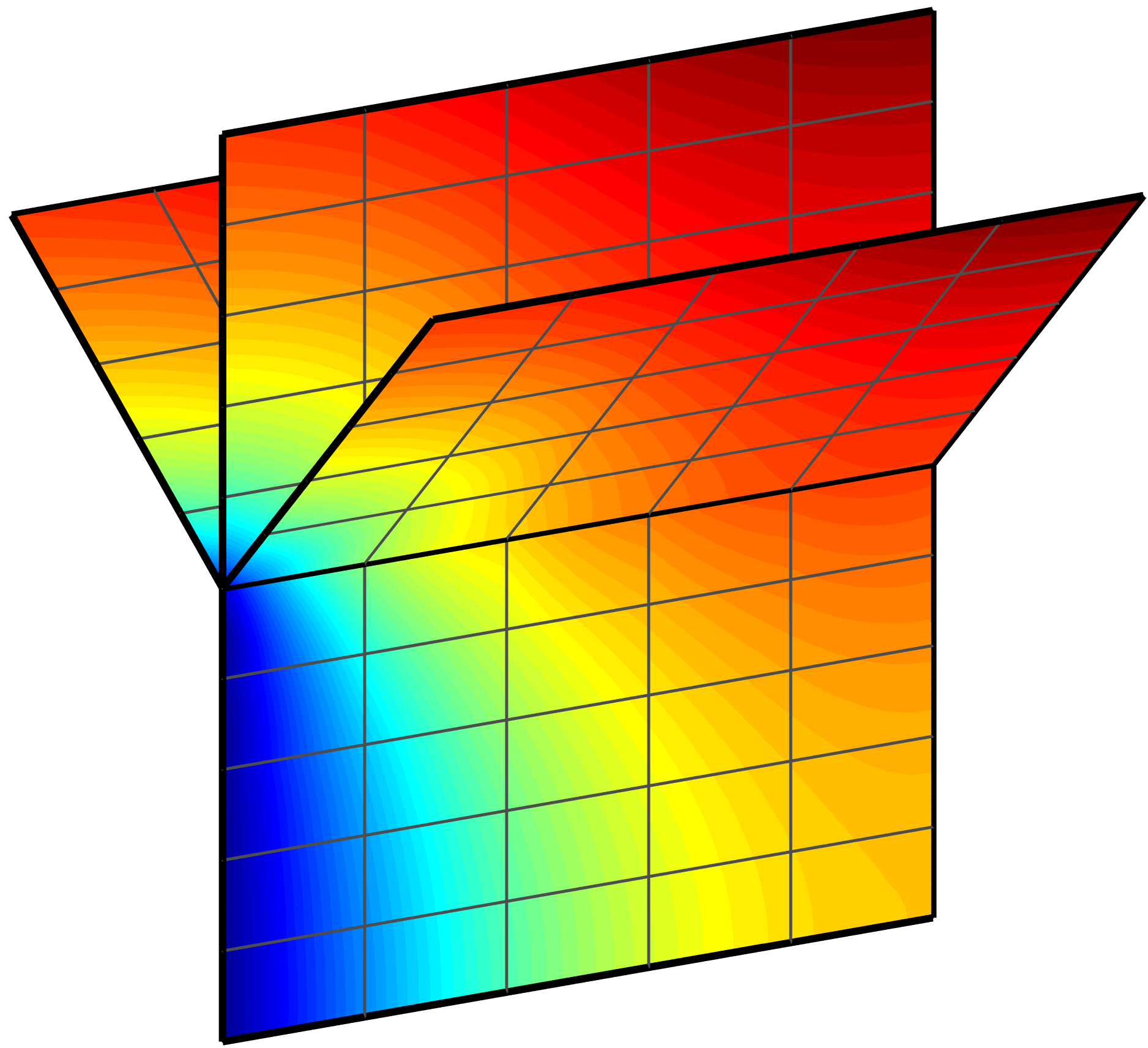}
  \end{subfigure}
    \begin{subfigure}[t]{0.24\linewidth}\centering
    \includegraphics[height=3.5cm]{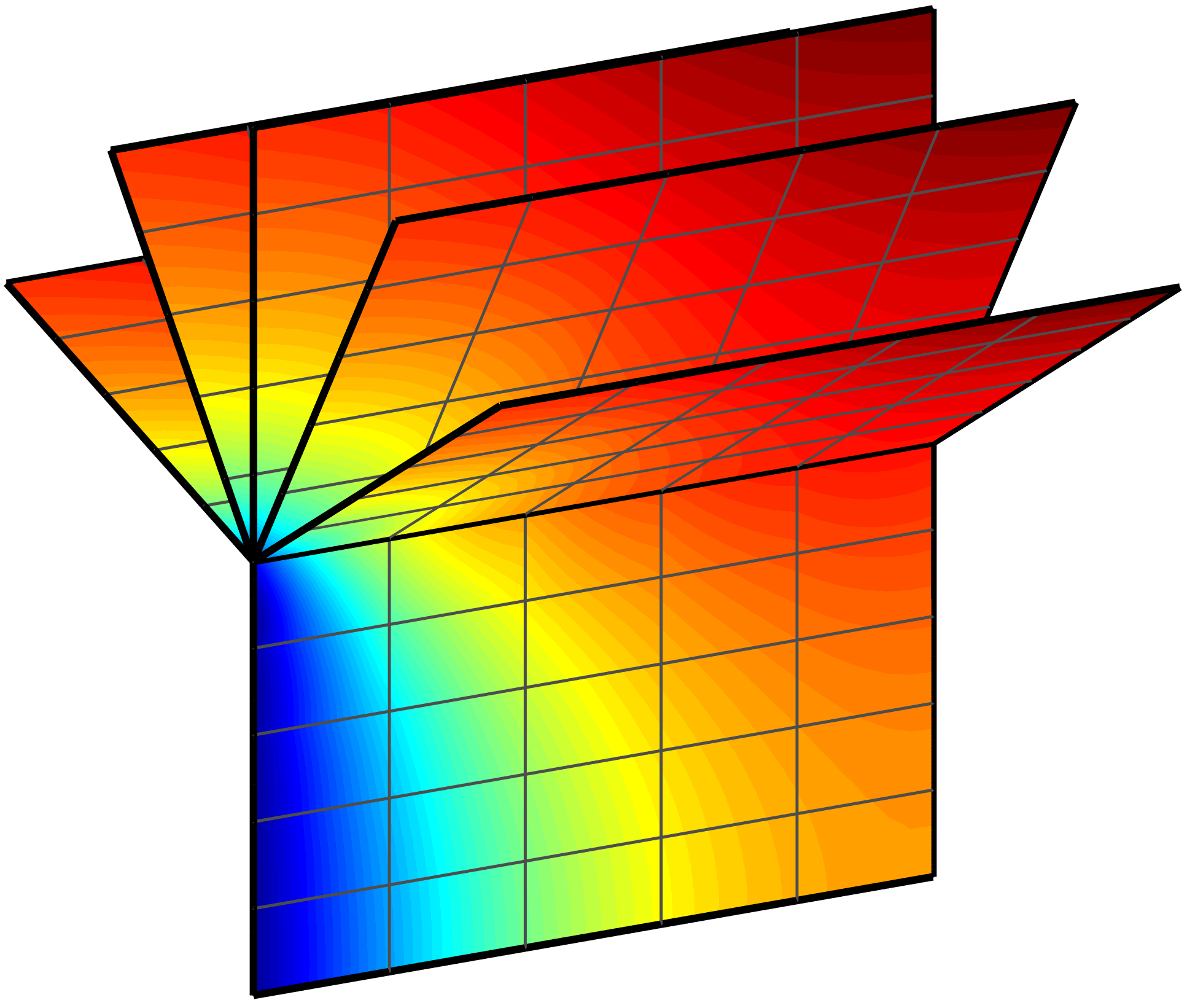}
  \end{subfigure}
  \begin{subfigure}[t]{0.24\linewidth}\centering
    \includegraphics[height=3.5cm]{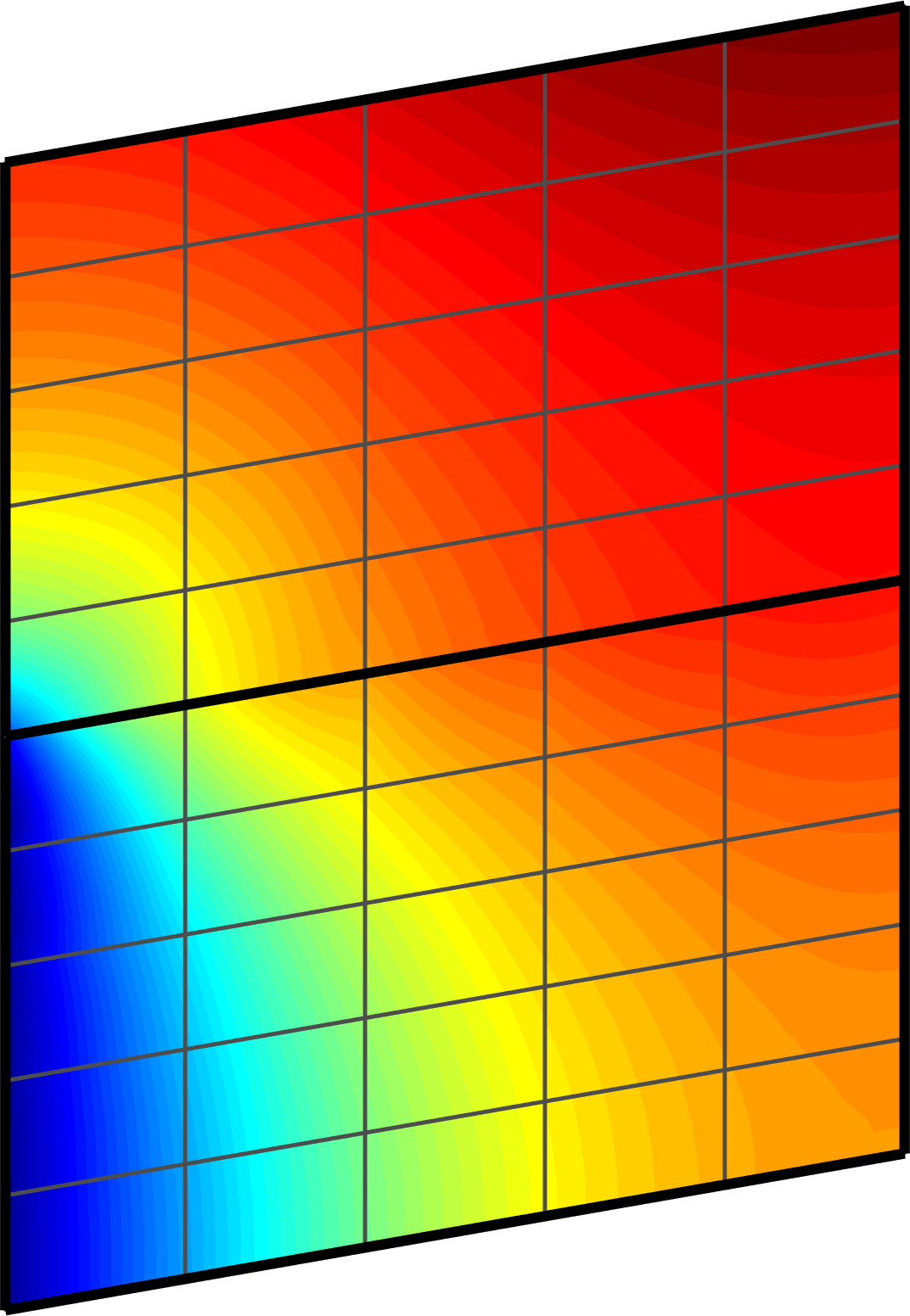}
  \end{subfigure}
  \begin{subfigure}[t]{0.24\linewidth}\centering
    \includegraphics[height=3.5cm]{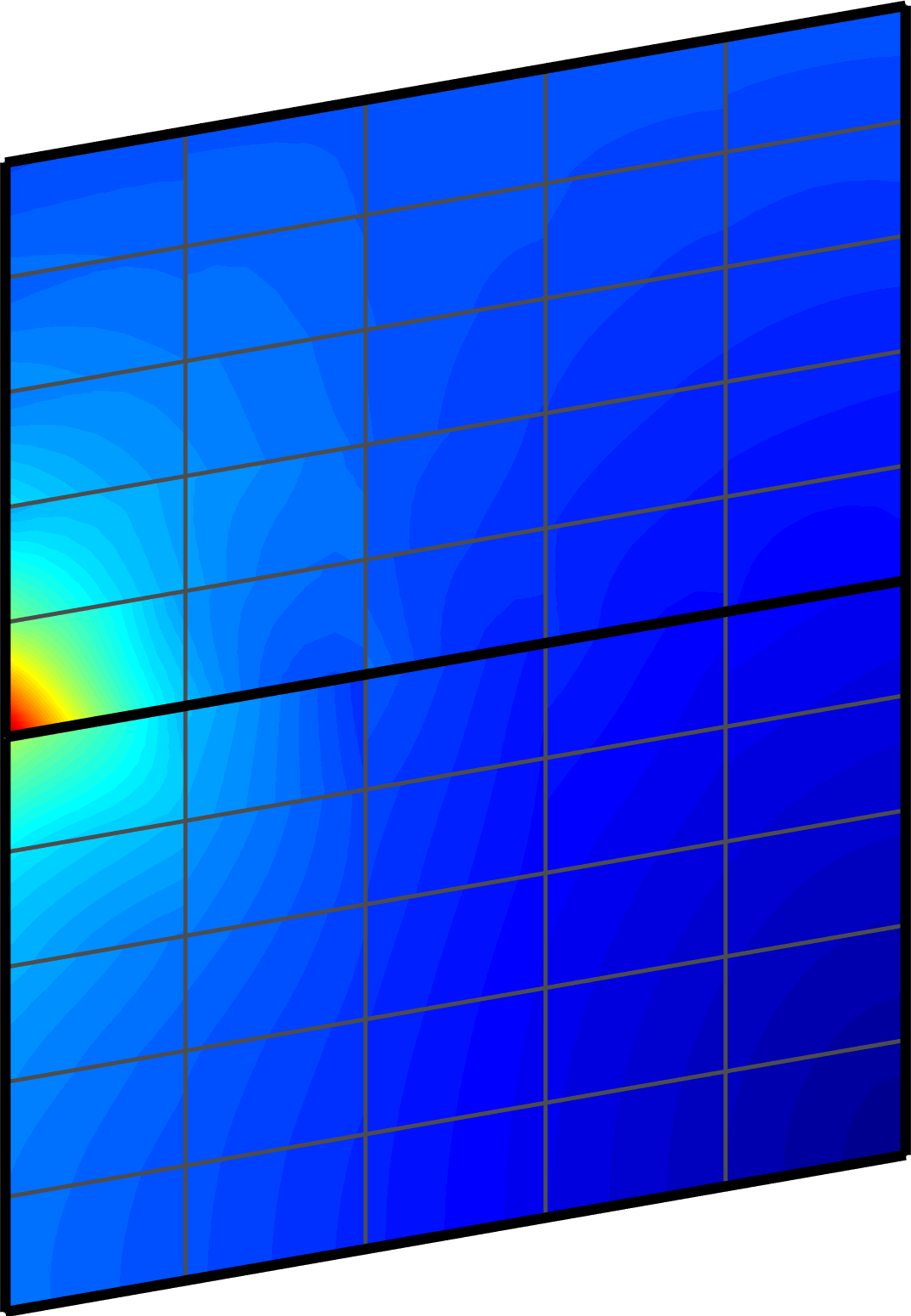}
  \end{subfigure}
  \begin{subfigure}[t]{0.24\linewidth}\centering
    \includegraphics[height=3.5cm]{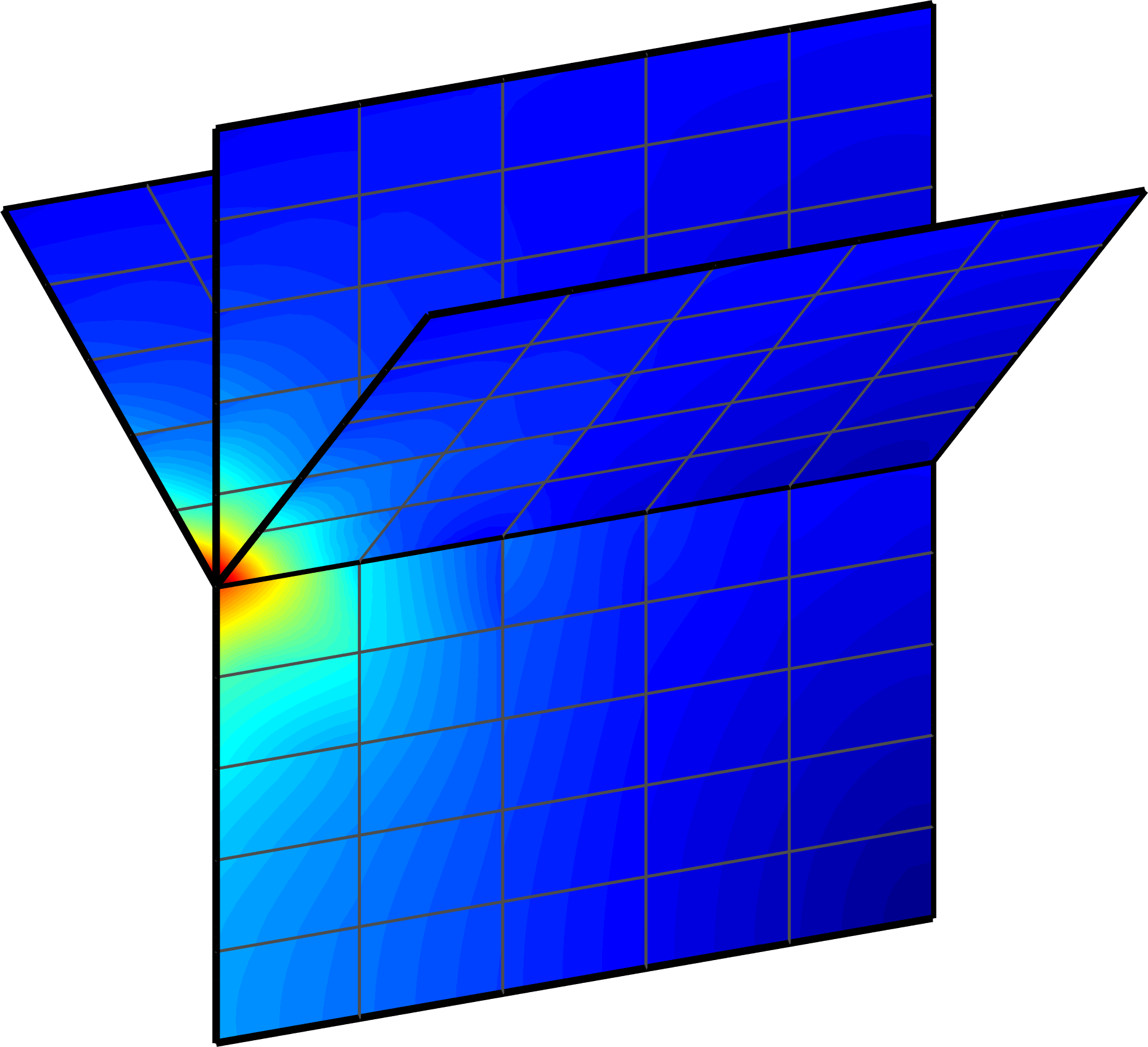}
  \end{subfigure}
    \begin{subfigure}[t]{0.24\linewidth}\centering
    \includegraphics[height=3.5cm]{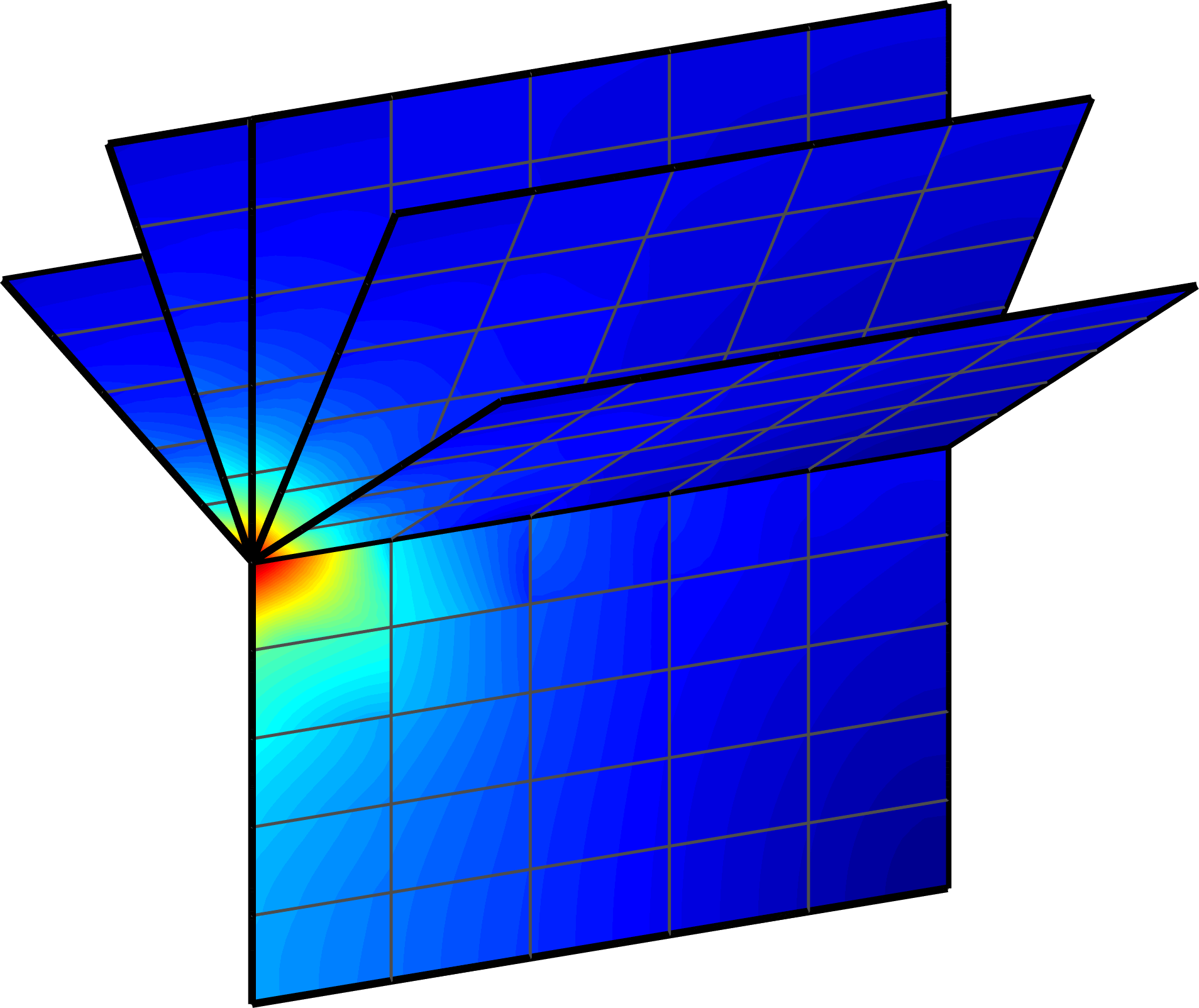}
  \end{subfigure}
  \begin{subfigure}[t]{0.24\linewidth}\centering
    \includegraphics[height=3.5cm]{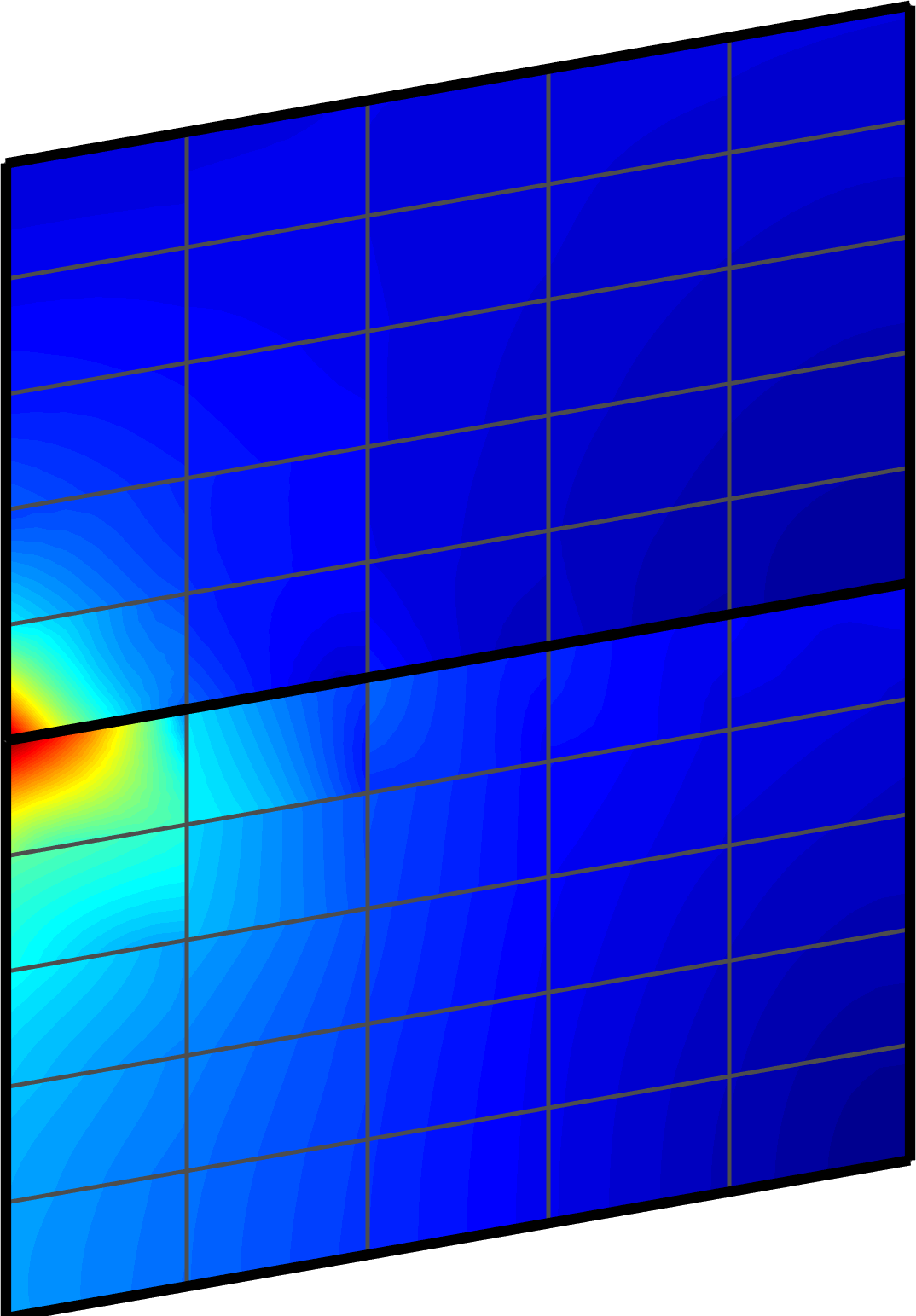}
  \end{subfigure}
  \caption{Numerical investigation of the gradient singularity occurring when switching between Dirichlet and Neumann boundary conditions exactly at the interface. On the left side of the bottom card we impose $u=0$ and on the top of each second row card we impose $\nu\cdot\sigma(u)=1/n$ where $n \in \{1,3,5\}$ is the number of cards added. Thus, the total flux over the top boundaries is the same in each set-up. The top subfigures display the solution and the bottom subfigures display the gradient magnitude. In the rightmost example we have $\mu=5$ in the top card and $\mu=1$ in the bottom card. Note that as we increase the number of cards the gradient singularity in the bottom card becomes more profound. Comparing the two examples on the right we see the very same behaviour in the bottom card when adding $n=5$ cards as is achieved when adding a single card with material coefficient $\mu=n$.}
\label{fig:singularity-study}
\end{figure}


\subsection{Convergence}
\paragraph{Intersecting Cylinders.}
Here we consider the composite surface constructed by intersecting two cylinders with radii $R_1<R_2$ and cutting the surfaces along the intersection as shown in Figure~\ref{fig:cylinders-domain}. The cylinder axes are parallel to the $xy$-plane, their relative rotation in the $xy$-plane is $\pi/2+\theta$, $0\leq\theta\leq \pi/2$, and their relative offset in $z$-direction is $z_0$. The points on each cylinder thus satisfy
\begin{align}
y^2 + (z-z_0)^2 = R_1^2 \quad\text{and}
\quad
(x\sin(\theta)+y\cos(\theta))^2+z^2 = R_2^2
\end{align}
respectively.
The cylinder intersection satisfy both these equations which after simplification gives that the intersection is described by
\begin{align}
(x,y,z)\in\IR^3:\left\{
\begin{aligned}
 x &= \pm \sqrt{R_1^2-z^2} 
 \\
 y &= x\cot(\theta) \pm \csc(\theta)\sqrt{R_2^2-z^2+2z_0 z-z_0^2}
\end{aligned}
 \ , \quad
|z-z_0| \leq R_2 
\right.
\end{align}
To describe the example geometry displaced in Figure~\ref{fig:cylinders-domain}
we use parameters $R_1=1.5$, $R_2=2$, $\theta=2\pi/3$ and $z_0=0.15$.

We manufacture a problem with known solution by choosing a solution $u$ which is analytical on each subsurface $\Omega_i$ and generating the appropriate right hand side $f=-\Delta_\Omega u$ and matching non-homogeneous Dirichlet boundary conditions.
For any function $u$ which have a smooth extension to a volumetric neighbourhood to $\Omega_i$ the Laplace--Beltrami operator $\Delta_{\Omega}$ can be expressed in terms of the usual $\IR^3$ Laplacian $\Delta$ and the first and second normal derivatives as 
\begin{equation}
\Delta_{\Omega} u = \Delta u - \partial_{nn} u -  2H \partial_n u
\end{equation}
where $H = \nabla \cdot n$ is the mean curvature of $\Omega_i$.
In the present example we choose the solution as the restriction of
\begin{equation}
u = \sin(x) \sin(y) \sin(z)
\label{eq:manufactured-solution}
\end{equation}
to $\Omega$ and by construction \eqref{eq:manufactured-solution} also is an extension of $u$ to $\IR^3$. This solution and the magnitude of its gradient is presented in Figure~\ref{fig:cylindersol_split} where we separated the subsurfaces to show the solution also on the subsurfaces hidden from view. Note that this solution will also satisfy the Kirchhoff condition \eqref{eq:prob-interface-flux} as $u$ is smooth and that for each subsurface in an interface there is a subsurface with exactly opposing conormal. 

Using this manufactured problem we study the convergence of the method \eqref{eq:fem} in $L^2(\Omega)$ norm with $p=1$ and $p=2$ finite elements. An example finite element solution using $p=2$ elements is presented in Figure~\ref{fig:cylindersol} and the convergence results are presented in Figure~\ref{fig:L2err} where we note that we achieve optimal order convergence of $\mathcal{O}(h^{p+1})$. Looking at the $L^2$ estimate \eqref{eq:error-estimate-L2} this is expected as the manufactured solution $u$ is smooth by definition and each subsurface in the intersecting cylinders composite surface has a smooth boundary which should give the necessary regularity for the dual solution $\phi$.

Using the same problem set-up we also investigate how the condition number of the stiffness matrix scales with $h$ for $p=1$ elements. As shown in Figure~\ref{fig:condnum} we get the expected scaling of $\mathcal{O}(h^{-2})$.

\begin{figure}[p]
\centering
  \begin{subfigure}[b]{0.45\textwidth}
    \includegraphics[width=\textwidth]{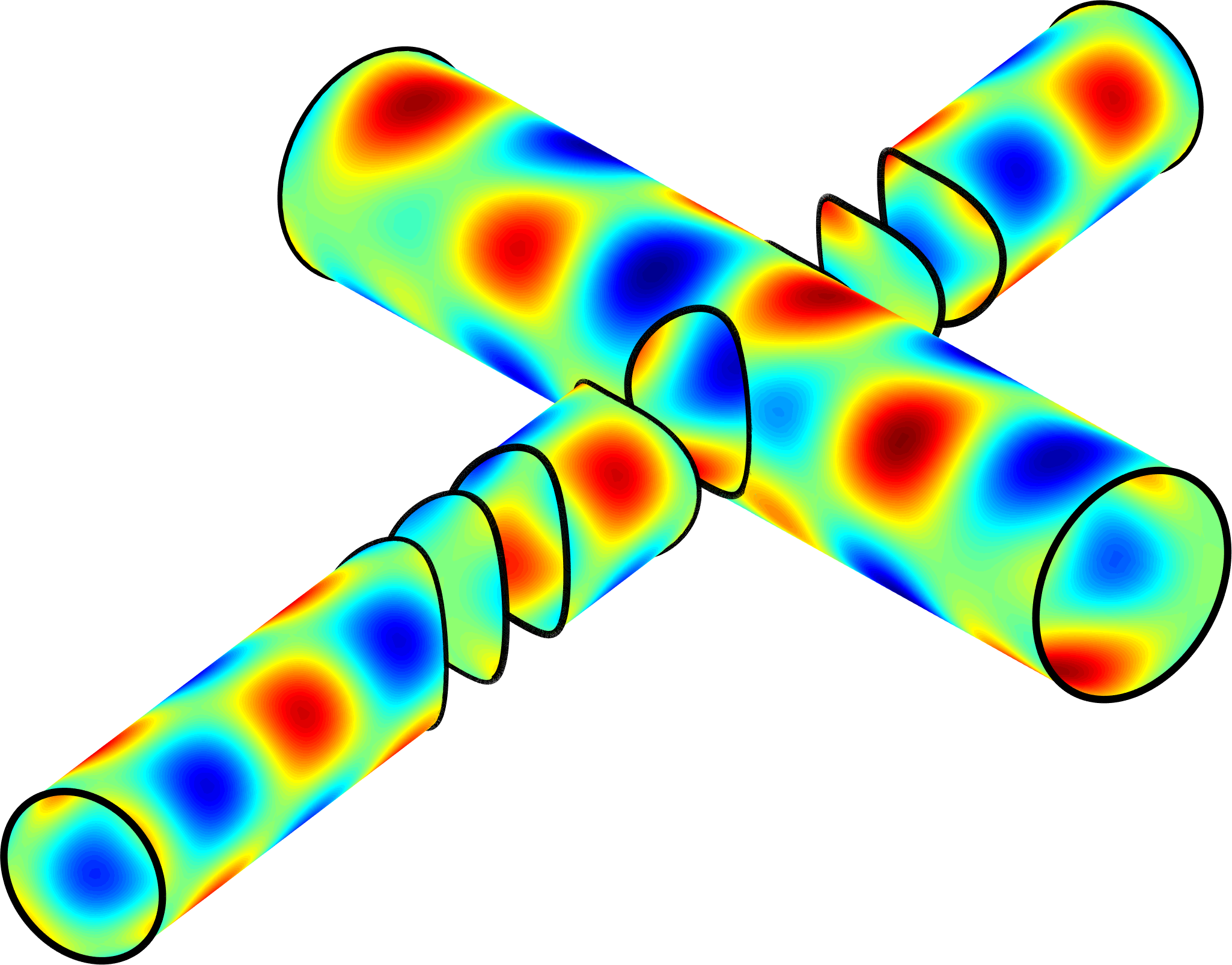}
   \caption{Solution}
  \end{subfigure}
  \quad
  \begin{subfigure}[b]{0.45\textwidth}
    \includegraphics[width=\textwidth]{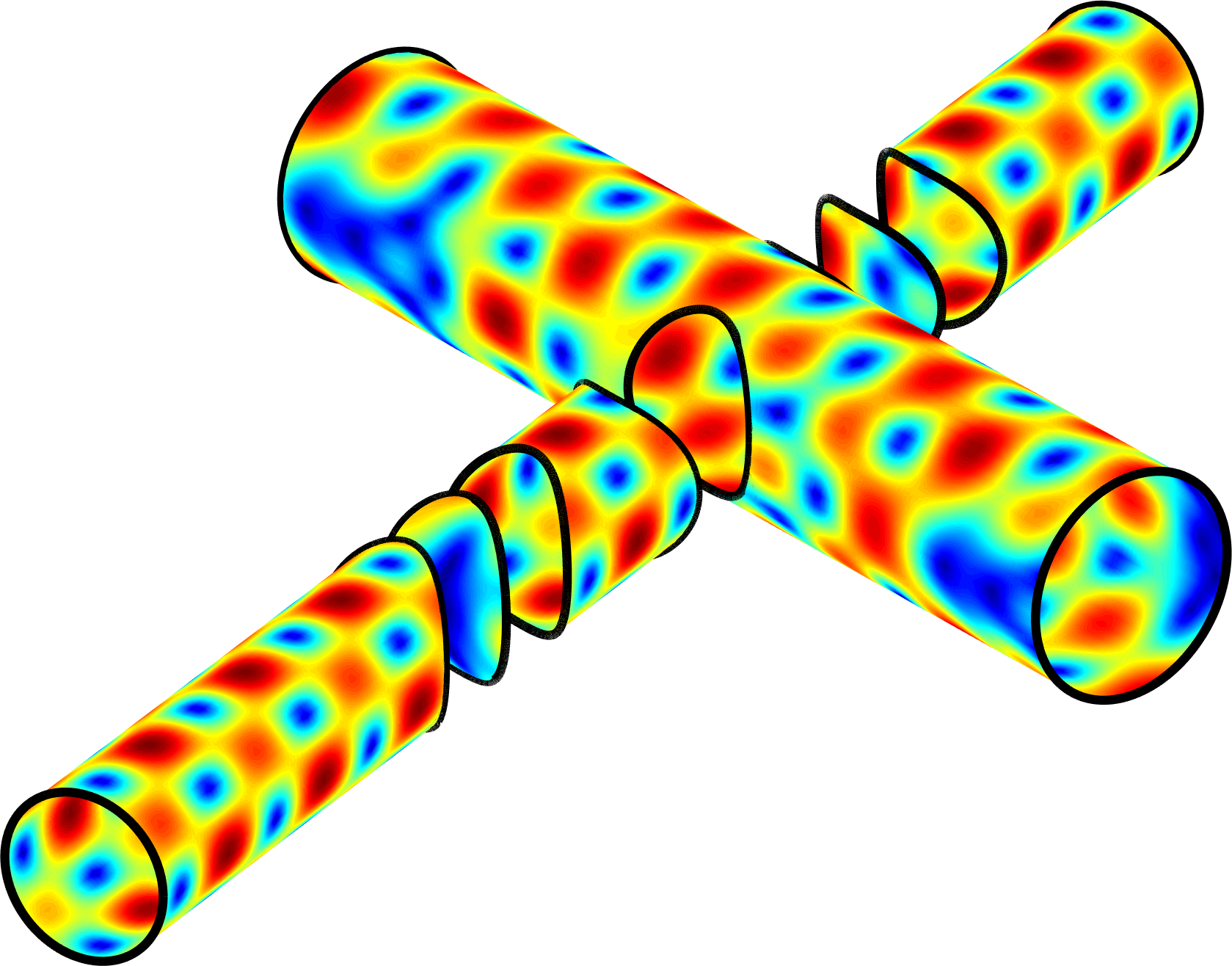}
  \caption{Magnitude of the gradient}
  \end{subfigure}
   \caption{The manufactured solution $u$ and the surface gradient on the intersecting cylinders with the pieces moved apart.}
  \label{fig:cylindersol_split}
\end{figure}

\begin{figure}[p]
\centering
  \begin{subfigure}[b]{0.45\textwidth}
    \includegraphics[width=\textwidth]{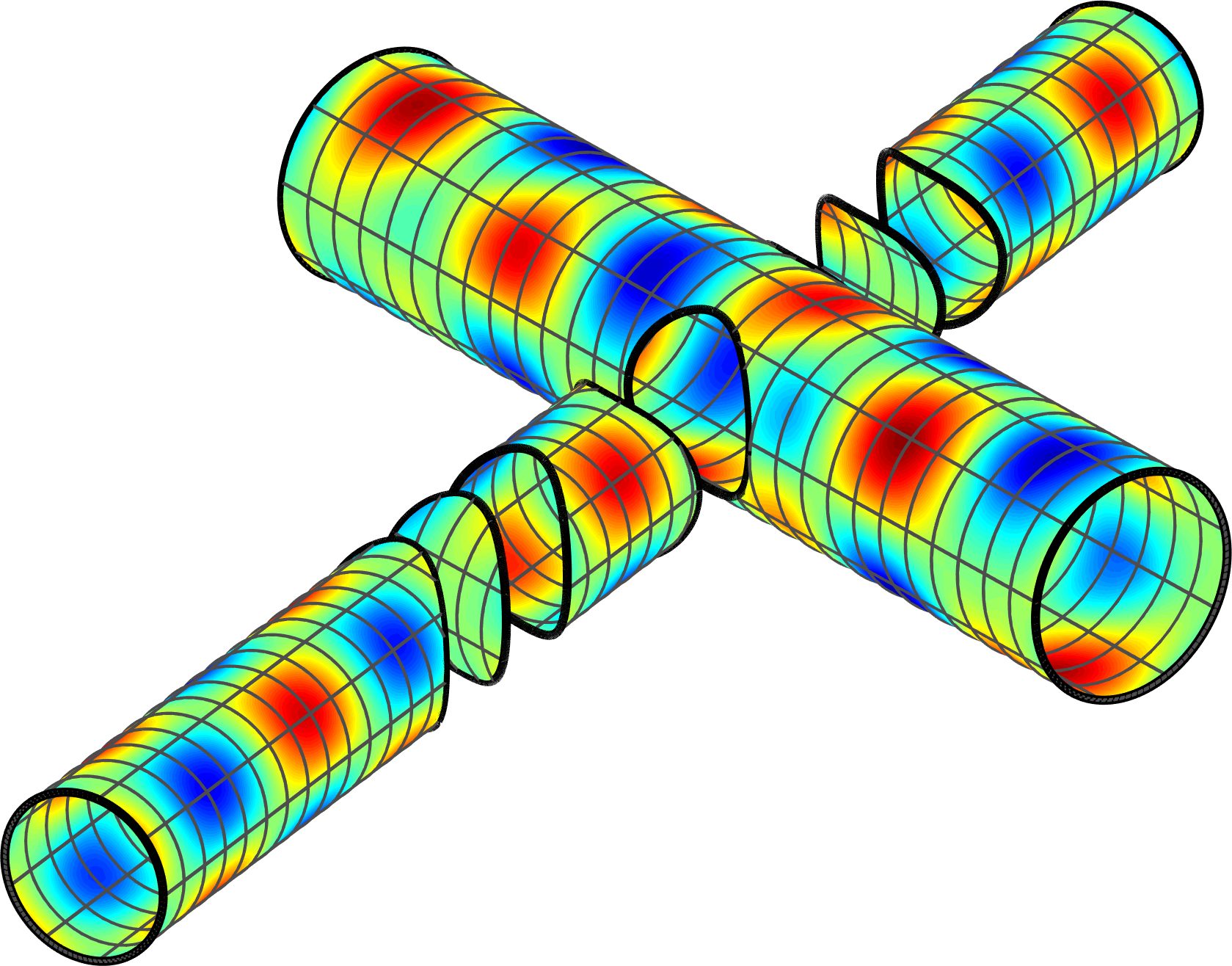}
   \caption{Solution}
  \end{subfigure}
  \quad
  \begin{subfigure}[b]{0.45\textwidth}
    \includegraphics[width=\textwidth]{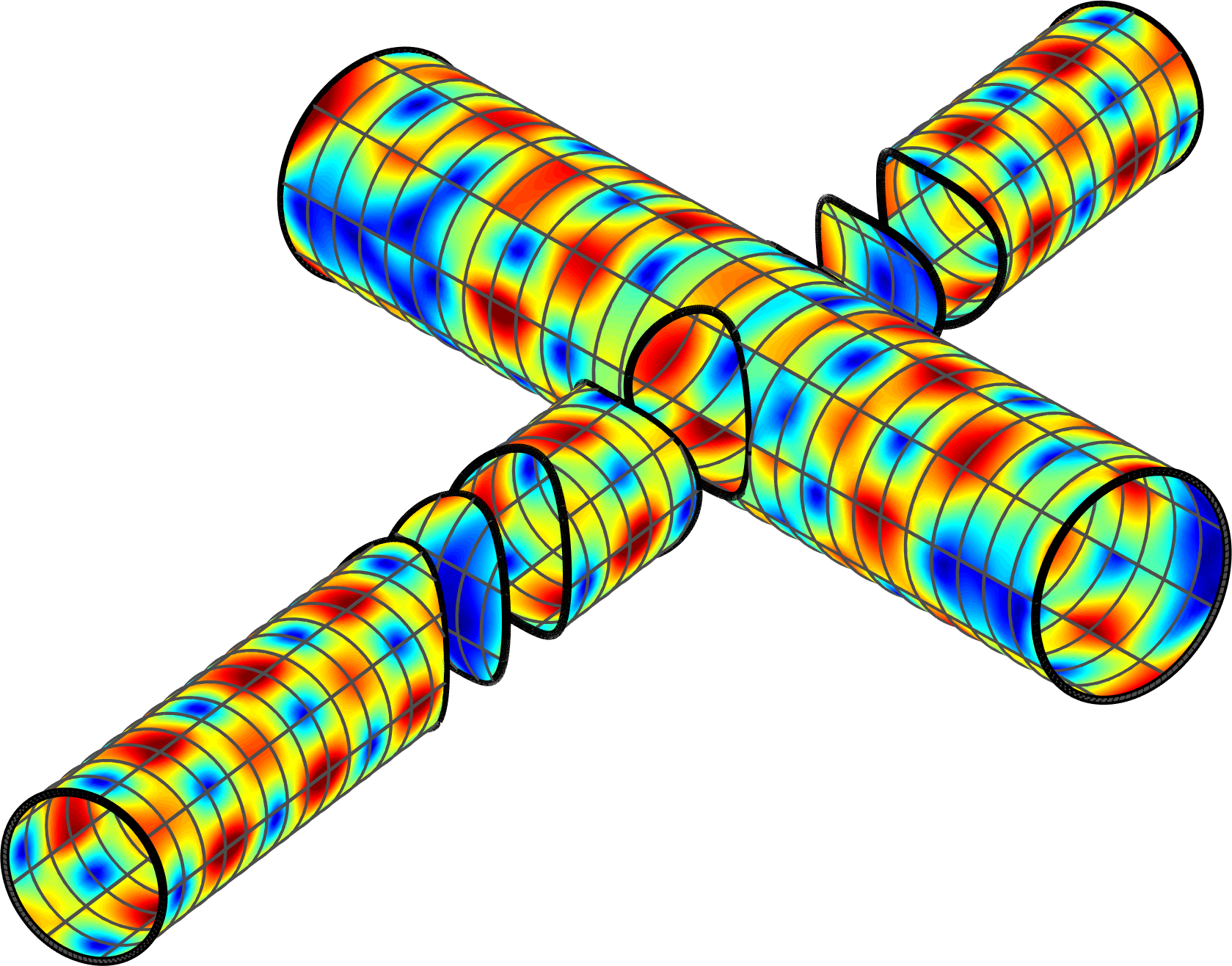}
  \caption{Magnitude of the gradient}
  \end{subfigure}
   \caption{Finite element solution to the intersecting cylinders manufatured problem.}
  \label{fig:cylindersol}
\end{figure}


\begin{figure}
\centering
  \includegraphics[width=0.5\textwidth]{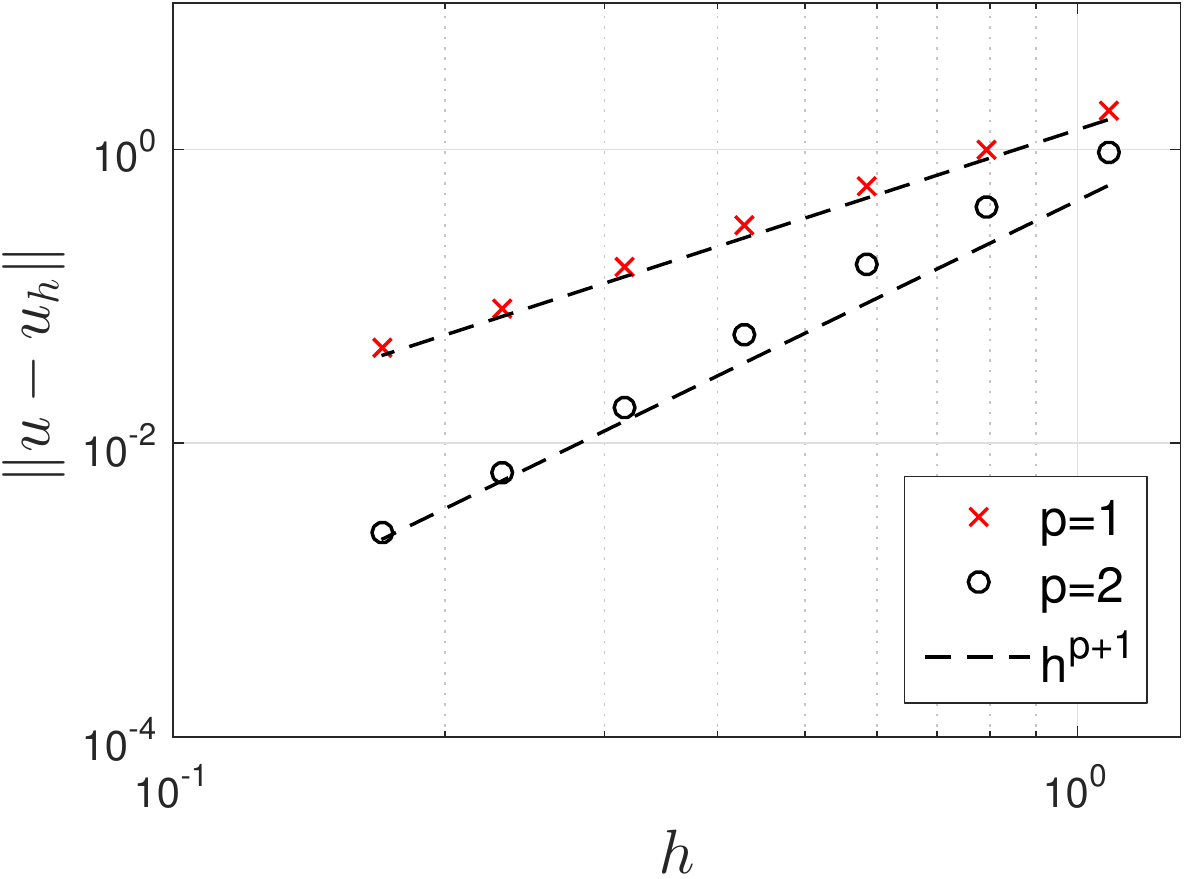}
  \caption{Convergence in $L^2$ norm for the manufactured problem on the intersecting cylinders composite surface. The dashed reference lines are $\mathcal{O}(h^{p+1})$.} 
  \label{fig:L2err}
\end{figure}

\begin{figure}
\centering
  \includegraphics[width=0.5\textwidth]{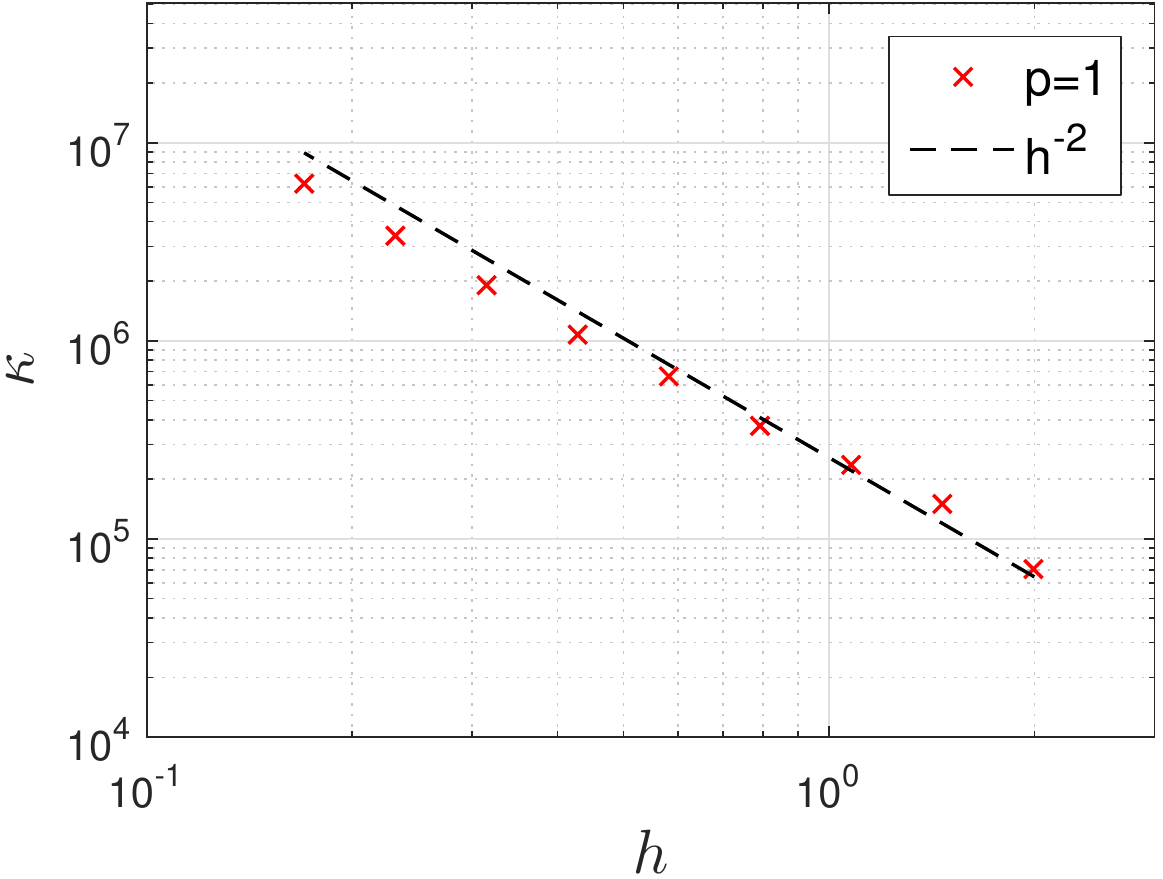}
  \caption{Scaling of the condition number for the stiffness matrix in the intersected cylinder problem using $p=1$ elements.}
  \label{fig:condnum}
\end{figure}

\clearpage

\bibliographystyle{abbrv}
\bibliography{ref}

\end{document}